\documentclass[reqno,english]{amsart}
\usepackage{amsfonts,amsmath,latexsym,verbatim,amscd,mathrsfs,color,array}
\usepackage[colorlinks=true]{hyperref}
\usepackage{amsmath,amssymb,amsthm,amsfonts,graphicx,color}
\usepackage{amssymb}
\usepackage{pdfsync}
\usepackage{epstopdf}
\usepackage{cite}

\usepackage{graphicx}

\newcommand{\cuad}{{\sqcap\kern-.68em\sqcup}}

\newcommand{\be}{\begin{equation}}
\newcommand{\ee}{\end{equation}}

\newcommand{\bs}{\begin{split}}
\newcommand{\esp}{\end{split}}

\newtheorem{lemma}{Lemma}[section]
\newtheorem{prop}{Proposition}[section]
\newtheorem{theorem}{Theorem}

\newtheorem{remark}{Remark}[section]
\newcommand{\bremark}{\begin{remark} \em}
\newcommand{\eremark}{\end{remark} }

\long\def\comment#1{\marginpar{\vtop{\raggedright\small$\bullet$\ #1}}}
\long\def\hide#1{}

\long\def\anot#1{\ \\{\bf \crr ANNOTATION.} {#1}}
\long\def\noanot#1{}
\definecolor{redd}{RGB}{200,0,0}

\def\crr{\color{red}}
\long\def\elim#1{{\color{red} ELIMINAR\\ #1}}

\def\crr{}
\long\def\comment#1{}
\long\def\anot#1{}
\long\def\elim#1{}

\long\def\comment#1{\marginpar{\raggedright\small$\bullet$\ #1}}
\numberwithin{equation}{section}
\title[fast diffusion equations with Dirichlet boundary conditions]{Extinction behaviour for the fast diffusion equations with critical exponent and Dirichlet boundary conditions}
\author[Y. Sire]{Yannick Sire}
\address{\noindent Department of Mathematics, Johns Hopkins University, 404 Krieger Hall, 3400 N. Charles Street, Baltimore, MD 21218, USA}
\email{sire@math.jhu.edu}
\author[J. Wei]{Juncheng Wei}
\address{\noindent Department of Mathematics, University of British Columbia, Vancouver, B.C., Canada, V6T 1Z2}
\email{jcwei@math.ubc.ca}
\author[Y. Zheng]{Youquan Zheng}
\address{\noindent School of Mathematics, Tianjin University, Tianjin 300072, P. R. China}
\email{zhengyq@tju.edu.cn}
\begin{document}
\begin{abstract}
For a smooth bounded domain $\Omega\subseteq\mathbb{R}^n$, $n\geq 3$, we consider the fast diffusion equation with critical sobolev exponent $$\frac{\partial w}{\partial\tau} =\Delta w^{\frac{n-2}{n+2}}$$ under Dirichlet boundary condition $w(\cdot, \tau) = 0$ on $\partial\Omega$. Using the parabolic gluing method, we prove existence of  an initial data $w_0$ such that the corresponding solution has extinction rate of the form $$\|w(\cdot, \tau)\|_{L^\infty(\Omega)} = \gamma_0(T-\tau)^{\frac{n+2}{4}}\left|\ln(T-\tau)\right|^{\frac{n+2}{2(n-2)}}(1+o(1))$$ as $t\to T^-$, here $T > 0$ is the finite extinction time of $w(x, \tau)$. This generalizes and provides rigorous proof of  a result of Galaktionov and King \cite{galaktionov2001fast} for the radially symmetric case $\Omega =B_1(0) : = \{x\in \mathbb{R}^n||x| < 1\}\subset\mathbb{R}^n$.
\end{abstract}
\maketitle

\tableofcontents

\section{Introduction}
Let $\Omega$ be a smooth bounded domain in $\mathbb{R}^n$, $n\geq 3$. We consider the following fast diffusion equation
\begin{equation}\label{e:fastdiffusion}
\begin{cases}
\frac{\partial w}{\partial\tau} =\Delta w^m \text{ in } \Omega\times (0, \infty),\\
w = 0\text{ on } \partial\Omega\times (0, \infty),\\
w(\cdot, 0) = w_0 \text{ in }\overline{\Omega},
\end{cases}
\end{equation}
with $m\in (0, 1)$. The first equation in (\ref{e:fastdiffusion}) is a singular but non-degenerate parabolic problem. From \cite{Kwong1998}, we know that there exists a unique positive classical solution $w$ which is local in time for the the Dirichlet problem (\ref{e:fastdiffusion}). The solution vanishes at finite time as $\tau\to T- < \infty$, $w > 0$ in $\Omega\times (0, T)$ and $w(x, T) = 0$.

The asymptotic behaviour of solutions for (\ref{e:fastdiffusion}) near the extinction time $T$ has attracted much attention in the past two decades. Suppose $\Omega =B_1(0) : = \{x\in \mathbb{R}^n||x| < 1\}\subset\mathbb{R}^n$, when $m\in (m_s, 1)$ and $m_s := \frac{n-2}{n+2}$. From the classical work of Berryman and Holland \cite{BerrymanHolland}, the solution near the extinction time has a separated self-similar form
$$
w(x, \tau) = (T-\tau)^{\frac{1}{1-m}}S(x),
$$
where $S(x)$ is the positive solution of the following nonlinear elliptic problem
$$
\Delta S^m + (1-m)^{-1}S = 0\text{ in }\Omega, \quad S=0\text{ on }\partial\Omega.
$$
When $m\in (0, m_s)$, it was proved in \cite{galaktionov1997}, \cite{galaktionov2001fast}, \cite{king1993self} and \cite{King2010asymptotic} that the self-similar behavior as $t\to T-$ can be described as
$$
w(x, \tau)\sim (T-\tau)^\alpha F\left(\frac{|x|}{(T-\tau)^\beta}\right), \quad (1-m)\alpha + 2\beta = 1,
$$
which provides the leading order of the inner solution. Thus the inner region is $|x| = O((T-\tau)^\beta)$ and the outer region is $|x| = O(1)$ with
$$
w(x, \tau)\sim (T-\tau)^{(m\alpha + (n-2)\beta)/m}\Phi(x),
$$
where $\Phi(x)$ is the Green's function with Dirichlet boundary condition,
$$
\Delta\Phi = -C_{n, m}\delta(x)\text{ in }\Omega, \quad \Phi=0\text{ on }\partial\Omega,
$$
where $C_{n, m}$ is a positive constant depending on $n$ and $m$, $\delta(x)$ is the Dirac delta distribution function locating at origin.

For general smooth bounded domains, the papers \cite{BerrymanHolland}, \cite{BonforteGrilloVazquezJMPA}, \cite{DibenedettoKwonga}, \cite{DibenedettoKwongb} and \cite{FeireslSimondon} studied the asymptotic behaviour near extinction time for solutions to (\ref{e:fastdiffusion}). Recently, Bonforte and Figalli proved the sharp extinction rates in \cite{BonforteFigalli2019} for the supercritical case $m\in (m_s, 1)$. Optimal regularity at the boundary for solutions to (\ref{e:fastdiffusion}) was proved by Jin and Xiong in \cite{JinXiongFastDiffusion} when $m \in [m_s, 1)$. We refer the interested readers to \cite{blanchet2009asymptotics}, \cite{BonforteDolbeaultPNAS2010}, \cite{bonforte2006global}, \cite{DaskalopoulosDelpino2018}, \cite{DaskalopoulosandKenig}, \cite{galaktionov2000asymptotics}, \cite{kim2006potential}, \cite{vazquez2017mathematical} and the references therein for more results on the asymptotic behavior of fast diffusion and porous medium equations.

The case $m = m_s$ corresponds to the Yamabe flow which describes the evolution of conformal metrics; there are many results in the literature under different settings. For the Dirichlet problem (\ref{e:fastdiffusion}), sharp asymptotic results are still missing. To the best of our knowledge, the only asymptotic result was due to Galaktionov and King \cite{galaktionov2001fast}. The aim of this paper is to provide a rigourous asymptotic analysis of (\ref{e:fastdiffusion}) near the extinct time $T$ for general smooth domain $\Omega$. Our result can be stated as follows.

Let $H(x, y)$ be the regular part of the Green's function on $\Omega$ with Dirichlet boundary condition, i.e., for fixed $y\in \Omega$, $H(x, y)$ satisfies
$\Delta_x H(x, y) = 0$ in $\Omega$, $H(x, y) = \frac{(n(n-2))^{\frac{n-2}{4}}}{|x-y|^{n-2}}\text{ for } x\in \partial\Omega$. Let $q_1,\cdots, q_k$ to be $k$ different but fixed points in $\Omega$. We define the following matrix,
\begin{equation}\label{e:matrix}
\mathcal G (q) = \left [ \begin{matrix} H(q_1,q_1) &  -G(q_1,q_2) &\cdots & -G(q_1,q_k) \\   -G(q_1,q_2) &  H(q_2,q_2) &  -G(q_2,q_3) \cdots & -G(q_3,q_k) \\ \vdots &  & \ddots& \vdots   \\  -G(q_1,q_k) &\cdots&  -G(q_{k-1},q_k) &  H(q_k,q_k)            \end{matrix}    \right ].
\end{equation}
Our main result is
\begin{theorem}\label{t:mainforfastdiffusion}
Suppose $m = m_s = \frac{n-2}{n+2}$, $n\geq 3$, $T>0$ is the finite extinction time, $k$ is a positive integer and $q_1,\cdots, q_k$ are $k$ different but fixed points in $\Omega$ such that the matrix defined in (\ref{e:matrix}) is positive definite, then there exist an initial data $w_0$ and smooth functions $\tilde\mu_j(\tau)$, $\tilde\xi_j(\tau)$ such that the solution $w(x, \tau)$ of problem (\ref{e:fastdiffusion}) has the following asymptotic form when $\tau\to T-$,
\begin{equation*}
\begin{aligned}
&w^{\frac{n-2}{n+2}}(x, \tau) = \left(T-\tau\right)^{\frac{n-2}{4}}\times\\
&\quad\quad\quad\quad\left(\sum_{j=1}^k\left(\alpha_n\left(\frac{\tilde\mu_j(\tau)}{\tilde\mu^2_j(\tau)+|x-\tilde\xi_j(\tau)|^2}\right)^{\frac{n-2}{2}} - \tilde\mu^{\frac{n-2}{2}}_j(\tau)H(x, q_j)\right)+\tilde\varphi(x, \tau)\right),
\end{aligned}
\end{equation*}
where the parameters $\tilde\mu_j(\tau) = \beta_j \left(\log\frac{T}{T-\tau}\right)^{-\frac{1}{n-2}}(1+o(1))$ for some $\beta_j > 0$, $\tilde\xi_j-q_j = o\left(\left(\log\frac{T}{T-\tau}\right)^{-\frac{1}{n-2}}\right)$, $\alpha_n = (n(n-2))^{\frac{n-2}{4}}$ and $\tilde\varphi(x, \tau) \to 0$ uniformly away from the points $q_1,\cdots, q_k$ as $\tau\to T-$.
\end{theorem}

In the paper \cite{galaktionov2001fast}, Galaktionov and King gave the extinction rate $\|w(\cdot, \tau)\|_{\infty} = \gamma_0(T-\tau)^{\frac{n+2}{4}}\left|\ln(T-\tau)\right|^{\frac{n+2}{2(n-2)}}(1+o(1))$ when $\Omega =B_1(0) : = \{x\in \mathbb{R}^n||x| < 1\}\subset\mathbb{R}^n$ by matching expansions from the inner and boundary domains. Theorem \ref{t:mainforfastdiffusion} gives a rigourous proof of this extinction rate as well as a description of the space part in the multiple point case for general domains. We refer the interested readers to \cite{king1993self} and \cite{king1993exact} for more results on the extinction behaviour of the fast diffusion equations.

 In the inner region near the point $q_j$, $w(x, \tau)$ is a logarithmic perturbation of the self-similar stationary structure. Indeed, we have
$$
w(x, \tau) = (T-\tau)^{\frac{n+2}{4}}\alpha(t)S_1(|x-q_j|\alpha^{\frac{2}{n+2}}(\tau))(1+o(1))
$$
with $\alpha(\tau) = \gamma_0\left(\log\frac{T}{T-\tau}\right)^{\frac{n+2}{2(n-2)}}$ and $S_1$ belongs to a one-parameter family of stationary positive solutions $\{S_\lambda(|x|)|\lambda >0\}$, which are the Loewner-Nirenberg explicit solutions
$$
S_\lambda(r)=\lambda\left[\frac{2n(n-2)}{2n(n-2)+(n+2)\lambda^{\frac{4}{n+2}}r^2}\right] = \lambda S_1(r\lambda^{\frac{2}{n+2}})
$$
to the nonlinear elliptic equation $\Delta S^{\frac{n-2}{n+2}} + \frac{1}{4}(n+2)S = 0\text{ in }\mathbb{R}^n$, see \cite{LoewnerNirenber}.

Under the transformation
\begin{equation}\label{e:transformation}
u(x, t) = (T-\tau)^{-m/(1-m)}w(x, \tau)^m|_{\tau = T(1-e^{-t})},
\end{equation}
Problem (\ref{e:fastdiffusion}) changes into the Yamabe flow equation on the bounded domain $\Omega$ as follows,
\begin{equation}\label{e:main}
\begin{cases}
\frac{\partial u^p}{\partial t}=\Delta u+u^p \text{ in } \Omega\times (0, \infty),\\
u = 0\text{ on } \partial\Omega\times (0, \infty),\\
u(\cdot, 0) = u_0 \text{ in }\overline{\Omega},
\end{cases}
\end{equation}
for a function $u:\mathbb{R}^n\times [0, \infty)\to \mathbb{R}$ and positive initial datum $u_0$ satisfying $u_0|_{\partial\Omega} = 0$, $p = \frac{n+2}{n-2}$. Therefore, using the transformation (\ref{e:transformation}), for problem (\ref{e:main}), Theorem \ref{t:mainforfastdiffusion} has the following equivalent form.
\begin{theorem}\label{t:main}
Suppose $n\geq 3$, $k$ is a positive integer and $q_1,\cdots, q_k$ are $k$ different but fixed points in $\Omega$ such that the matrix defined in (\ref{e:matrix}) is positive definite, then there exist an initial data $u_0$ and smooth functions $\mu_j(t)$, $\xi_j(t)$ such that the solution of problem (\ref{e:main}) has the following asymptotic form when $t\to +\infty$,
\begin{equation}\label{infinitetimeblowup}
u(x, t) = \sum_{j=1}^k\left(\alpha_n\left(\frac{\mu_j(t)}{\mu^2_j(t)+|x-\xi_j(t)|^2}\right)^{\frac{n-2}{2}} - \mu^{\frac{n-2}{2}}_j(t)H(x, q_j)\right)+\varphi(x, t),
\end{equation}
where $\mu_j  = \beta_j t^{-\frac{1}{n-2}}(1+o(1))$ for some $\beta_j > 0$, $\xi_j-q_j = o(t^{-\frac{1}{n-2}}) $, $\alpha_n = (n(n-2))^{\frac{n-2}{4}}$ and $\varphi(x, t) \to 0$ uniformly away from the points $q_1,\cdots, q_k$ as $t\to +\infty$.
\end{theorem}

The behaviour of Yamabe flow was studied in \cite{ye1994global}, \cite{brendle2007convergence}, \cite{brendle2007short}, \cite{brendle2005convergence}, \cite{daskalopoulos2016type}, \cite{DaskalopoulosDelpino2018}, \cite{daskalopoulos2013extinction}, \cite{daskalopoulos2008extinction}, \cite{daskalopoulos2013classification}, \cite{struwe} (see also \cite{jin2014fractional} for a related flow). Especially, in the case of $\mathbb{S}^n$ with its standard Riemannian metric $g_{\mathbb{S}^n}$, the Yamabe flow evolving a conformal metric $g = v^{\frac{4}{n-2}}(\cdot, t)g_{\mathbb{S}^n}$ takes the following form
\begin{equation}\label{Yamabeflow}
(v^{\frac{n+2}{n-2}})_t = \Delta_{\mathbb{S}^n}v - c_n v,\quad c_n = \frac{n(n-2)}{4},
\end{equation}
which is equivalent to the problem
\begin{equation}\label{e:fastdiffusiononthewholespace}
\begin{cases}
\frac{\partial}{\partial t} u^{\frac{n+2}{n-2}} =\Delta u + u^{\frac{n+2}{n-2}} \text{ in } \mathbb{R}^n\times (0, \infty),\\
u(\cdot, 0) = u_0 \text{ in }\mathbb{R}^n
\end{cases}
\end{equation}
via the stereographic projection and cylindrical changes of variables. It was proved in \cite{ye1994global}, \cite{brendle2007short} that the Yamabe flow (\ref{Yamabeflow}) has a global solution, which converges exponentially to a steady solution. In \cite{delpinoSaezIUMJ2001}, del Pino and Saez showed that solutions for problem (\ref{e:fastdiffusiononthewholespace}) approach non-trivial steady states of the semilinear elliptic equation
$$
\Delta u + u^{\frac{n+2}{n-2}} = 0 \text{ on }\mathbb{R}^n.
$$
Theorem \ref{t:main} tells us that when we consider the Yamabe flow equation on a bounded domain with Dirichlet boundary condition, infinite time blow-up phenomenon can occur.

In the beautiful work \cite{DaskalopoulosDelpino2018}, Daskalopoulos, del Pino and Sesum constructed a new class of type II ancient solutions to the  Yamabe flow; these solutions are rotationally symmetric and converge to a tower of spheres when $t\to -\infty$. Note that Theorem \ref{t:main} is on a bounded domain with Dirichlet boundary condition and  the solutions we find blow up at different points when the time $t\to +\infty$. In the recent paper \cite{delpinoMussoWei}, bubble tower solutions for the energy critical heat equation were constructed; we conjecture that bubble tower solutions for Problem (\ref{e:main}) also exist.

Infinite time blowing-up solutions for the energy critical heat equation with Dirichlet boundary condition
\begin{equation*}
\begin{cases}
\frac{\partial}{\partial t} u =\Delta u + u^{\frac{n+2}{n-2}} \text{ in } \Omega\times (0, \infty),\\
u(\cdot, t) = 0 \text{ on }\partial\Omega,\\
u(\cdot, 0) = u_0 \text{ in }\Omega
\end{cases}
\end{equation*}
of form (\ref{infinitetimeblowup}) are constructed in the seminal work \cite{delPinoMussoJEMS} when $n\geq 5$. Note that the corresponding blow-up rates are $\mu_j(t)\sim b_jt^{-\frac{1}{n-4}}(1+o(1))$ as $t\to +\infty$.

To prove Theorem \ref{t:main}, we use the gluing method in the spirit of \cite{delPinoMussoJEMS} and \cite{Davila2019}, which has been applied to various parabolic problems in recent years, such as finite time and infinite time blow-up solutions for energy critical heat equations \cite{delPinoMussoJEMS}, \cite{del2019type}, \cite{delpinoMussoWei}, \cite{delpinoMussoWei11}, \cite{delpinoMussoWei22}, ancient solutions of the Yamable flow \cite{DaskalopoulosDelpino2018}, singularity formation for the harmonic map heat flow \cite{Davila2019} and so on. In the survey paper by del Pino \cite{delpinosurvey}, there are  more results on the gluing method and its applications.

In the proof of Theorem \ref{t:main}, we first construct an approximation to the exact solution with sufficiently small error, then, by linearization around the bubble and fixed point theorem, we solve for a small remainder term. In the linear theory, we use blow-up arguments; the main difficulty is that the parabolic problem is degenerate, the linear equation is lifted onto the standard sphere $\mathbb{S}^n$, which then becomes a non-degenerate parabolic equation. Finally, based on the linear theory, we solve the nonlinear problem by the contraction mapping theorem. The orthogonality conditions are satisfied by solving an ODE system of the scaling and translation parameter functions.

\begin{remark}
The spectrum of the following degenerate elliptic operator
\begin{equation*}
L_0[\phi] = -\frac{1}{U^{p-1}}\left(\Delta\phi + pU^{p-1}\phi\right)
\end{equation*}
plays an important role in the linear theory. Since there is a negative eigenvalue for $L_0$ with multiplicity one (see Section 2), our solution constructed in Theorem \ref{t:main} is unstable. Indeed, from the proof the Theorem \ref{t:main} and the same arguments as in \cite{delPinoMussoJEMS}, there exists a submanifold $\mathcal{M}$ in the function space $X:=\{u\in C^1(\overline\Omega):u|_{\partial\Omega} = 0\}$ with codimension $k$ and containing $u_q(x, 0)$ such that, if $u_0$ is a small perturbation of $u_q(x,0)$ in $\mathcal{M}$, then the corresponding solution $u(x, t)$ to (\ref{e:main}) still has the asymptotic form
\begin{equation*}
u(x,t) =\sum_{j= 1}^k\left(\alpha_{n}\left(\frac{\hat{\mu}_j(t)}{\hat{\mu}_j^2(t) + |x-\hat{\xi}_j(t)|^2}\right)^{\frac{n-2}{2}}-\hat{\mu}_j^{\frac{n-2}{2}}(t)H(x, \hat{q}_j)\right)+\hat{\mu}_j^{\frac{n-2}{2}}(t)\hat{\varphi}(x, t),
\end{equation*}
the points $\hat{q}_j$ are close to $q_j$ for $j = 1,\cdots,k$. This is different from the ancient solution case; the effect of the negative eigenvalue can be dealt with by adding an additional parameter function which tends to 0 as $t\to -\infty$ (tends to $+\infty$ as $t\to +\infty$), see \cite{DaskalopoulosDelpino2018}.
\end{remark}

The paper is organized as follows: in Section 1, we build the approximate solution and provide a sketch of the inner-outer parabolic gluing, which gives a road map towards the proofs of our main statements. Section 2 is concerned with the outer problem whereas Section 3 is devoted to the inner one. Section 3 finishes with the complete proof of our results. 

\section{The approximate solution and the inner-outer gluing scheme}
\subsection{The approximate solution}
Let $t_0 > 0$ be a large number to be chosen later and consider the following problem
\begin{equation}\label{e:main000}
\begin{cases}
(u^p)_t=\Delta u+u^p \text{ in } \Omega\times (t_0, \infty),\\
 u = 0\text{ on } \partial\Omega\times (t_0, \infty),
\end{cases}
\end{equation}
for $p = \frac{n+2}{n-2}$. Let $q_1, \cdots, q_k\in\mathbb{R}^n$ be $k$ fixed points, we are going to find a positive solution to (\ref{e:main000}) of form
\begin{equation*}
u(x, t) \approx \sum_{j=1}^kU_{\mu_j(t), \xi_j(t)}(x)
\end{equation*}
with $\xi_j(t)\to q_j$, $\mu_j(t)\to 0$ as $t\to\infty$ for all $j = 1,\cdots, k$ and $U_{\mu_j(t), \xi_j(t)}(x) = \mu_j(t)^{-\frac{n-2}{2}}U\left(\frac{x-\xi_j(t)}{\mu_j(t)}\right)$, $U(y) = \alpha_n\left(\frac{1}{1+|y|^2}\right)^{\frac{n-2}{2}}$, which then provides a solution $u(x, t) = u(x, t-t_0)$ to the original problem (\ref{e:main}). Denote the error operator as follows
\begin{equation*}
S(u):= -(u^p)_t+\Delta u + u^p.
\end{equation*}
Then we have
\begin{eqnarray*}
\begin{aligned}
S(U_{\mu_j(t), \xi_j(t)}) &= -\frac{\partial }{\partial t}U^p_{\mu_j,\xi_j}(x) = \mu^{-\frac{n+2}{2}}_jU(y_j)^{p-1}\left(\frac{\dot{\mu_j}}{\mu_j} Z_{n+1}(y_j)  + \frac{\dot{\xi_j}}{\mu_j}\cdot \nabla U(y_j)\right)\\
& = \mu^{-\frac{n+2}{2}-1}_jU(y_j)^{p-1}\left(\dot{\mu}_j Z_{n+1}(y_j)  + \dot{\xi}_j\cdot \nabla U(y_j)\right)
\end{aligned}
\end{eqnarray*}
for $y_j = \frac{x-\xi_j(t)}{\mu_j(t)}$. Since $u = 0$ on $\partial\Omega$, a natural better approximation than $\sum_{j=1}^kU_{\mu_j(t), \xi_j(t)}(x)$ should be
\begin{equation}\label{e2:3}
\tilde z(x, t) = \sum_{j=1}^k\tilde z_j(x, t)~\mbox{with}~ \tilde z_j(x, t):=U_{\mu_j,\xi_j}(x) -  \mu_j^{\frac{n-2}{2}} H_{\mu_j}(x, q_j).
\end{equation}
Here for fixed $y\in \Omega$, $H_{\mu_j}(x, y)$ satisfies
$\Delta_x H_{\mu_j}(x, y) = 0$ in $\Omega$, $H_{\mu_j}(x, y) = \frac{(n(n-2))^{\frac{n-2}{4}}}{(\mu_j^2+|x-y|^2)^{\frac{n-2}{2}}}\text{ for } x\in \partial\Omega$. Then from the equation satisfied by $U_{\mu_j(t), \xi_j(t)}(x)$ and the fact that $H_{\mu_j}(x, q)$ is a harmonic function, the error of $\tilde z$ is
\begin{equation}\label{e2:3333}
S(\tilde z) = -\sum_{i=1}^k\partial_t\tilde z_i^p + \left(\sum_{i=1}^k\tilde z_i\right)^p - \sum_{i=1}^kU^p_{\mu_i,\xi_i}.
\end{equation}
Moreover, by the same arguments as that of \cite{delPinoMussoJEMS}, for a fixed index $j$, in the region $|x-q_j|\leq \frac{1}{2}\min_{i\neq l}|q_i - q_l|$, set $x = \xi_j + \mu_j y_j$, there holds
\begin{equation*}
S[\tilde z] = \mu_j^{-\frac{n+2}{2}}(\mu_j E_{0j} + \mu_jE_{1j} + \mathcal{R}_j)
\end{equation*}
with
\begin{equation*}
\begin{aligned}
E_{0j}&= pU(y_j)^{p-1}\left[-\mu_j^{n-3}H(q_j, q_j) + \sum_{i\neq j}\mu_j^{\frac{n-2}{2}-1}\mu_i^{\frac{n-2}{2}}G(q_j, q_i)\right]\\
&\quad +\mu_j^{-2}\dot{\mu}_jpU(y_j)^{p-1}Z_{n+1}(y_j),
\end{aligned}
\end{equation*}
\begin{equation*}
\begin{aligned}
E_{1j} &= pU(y_j)^{p-1}\left[-\mu_j^{n-2}\nabla H(q_j, q_j)+ \sum_{i\neq j}\mu_j^{\frac{n-2}{2}}\mu_i^{\frac{n-2}{2}}\nabla G(q_j, q_i)\right]\cdot y_j\\
&\quad + \mu^{-2}_jpU(y_j)^{p-1}\dot{\xi}_j\cdot \nabla U(y_j)
\end{aligned}
\end{equation*}
and
\begin{equation*}
\mathcal{R}_j = \frac{\mu_0^{n}g}{1+|y_j|^{2}}+\frac{\mu_0^{n-2}\vec{g}}{1+|y_j|^{4}}\cdot(\xi_j-q_j) + \mu_0^{n+2}f + \mu_0^{n-1}\sum_{i=1}^k\frac{\dot{\mu}_if_i}{1+|y_j|^{4}} + \mu_0^{n}\sum_{i=1}^k\frac{\dot{\xi}_i\cdot\vec{f}_i}{1+|y_j|^{4}},
\end{equation*}
where $f$, $f_i$, $\vec{f}_i$, $g$ and $\vec{g}$ are smooth and bounded functions of $(y, \mu_0^{-1}\mu, \xi, \mu_jy_j)$. Here $H(x, y)$ is the regular part of the Green's function on $\Omega$ with Dirichlet boundary condition, i.e., for fixed $y\in \Omega$, $H(x, y)$ satisfies $\Delta_x H(x, y) = 0$ in $\Omega$, $H(x, y) = \frac{(n(n-2))^{\frac{n-2}{4}}}{|x-y|^{n-2}}$ for $x\in \partial\Omega$.

Suppose $u = \tilde z + \tilde \phi$ is the exact solution of (\ref{e:main000}) and write $\tilde{\phi}(x, t)$ in self-similar form around the point $q_j$,
\begin{equation}\label{e2:29}
\tilde{\phi}(x, t) = \mu_j^{-\frac{n-2}{2}}\phi\left(\frac{x-\xi_j}{\mu_j}, t\right).
\end{equation}
Then we have
\begin{equation}\label{e2:30}
\begin{aligned}
&0=\mu_j^{\frac{n+2}{2}}S[\tilde z+\tilde{\phi}]\\
&\quad = -pU^{p-1}(y)\partial_t\phi + \Delta_y\phi + pU(y)^{p-1}\phi + \mu_j^{\frac{n+2}{2}}S[\tilde z] + A[\phi]
\end{aligned}
\end{equation}
with $A[\phi]$ being a high order term. To improve the approximation error, we require $\phi(y, t)$ equals (at main order) to the solution $\phi_{0j}(y, t)$ of the following equation
\begin{equation}\label{e2:31}
-pU^{p-1}(y)\partial_t\phi_{0j} + \Delta_y\phi_{0j} + pU(y)^{p-1}\phi_{0j} = -\mu_j^{\frac{n+2}{2}}S[\tilde z]\text{ in }\mathbb{R}^n.
\end{equation}
Near the blow-up point $q_j$, equation (\ref{e2:31}) is mainly an elliptic problem of form
\begin{equation}\label{e2:32}
L_0[\phi]:= \frac{1}{U^{p-1}}\left(\Delta_y\phi + pU(y)^{p-1}\phi\right) = h(y)\text{ in }\mathbb{R}^n,\,\,\psi(y)\to 0\text{ as }|y|\to \infty.
\end{equation}
Consider the eigenvalue problem $L_0[\phi] + \lambda \phi = 0$ on the weighted space $L^2(U^{p-1}dx)$, which has an infinite sequence of eigenvalues
\begin{equation*}
\lambda_{0} < \lambda_1 = \cdots = \lambda_n = \lambda_{n+1} = 0 < \lambda_{n+2} < \lambda_{n+3} < \cdots,
\end{equation*}
the associated eigenfunctions $Z_j$, $j = 0, 1,\cdots$ constitute an orthonormal basis of $L^2(U^{p-1}dx)$. It is well known that $\lambda_0$ is simple and $Z_0(y) = U(y)$. We refer the interested readers to the well written paper \cite{DaskalopoulosDelpino2018} and \cite{BonforteFigalli2019} for more properties on this operator.
Therefore every bounded solution of $L_0[\phi] = 0$ in $\mathbb{R}^n$ is the linear combination of the functions
$$Z_1,\cdots, Z_{n+1},$$
where
\begin{equation*}
Z_i(y):= \frac{\partial U}{\partial y_i}(y),\quad i = 1,\cdots, n, \quad Z_{n+1}(y):=\frac{n-2}{2}U(y) + y\cdot\nabla U(y).
\end{equation*}
Furthermore, problem (\ref{e2:32}) is solvable if the following conditions
\begin{equation*}
\int_{\mathbb{R}^n}h(y)Z_i(y)U^{p-1}(y)dy = 0\quad\text{for all}\quad i = 1,\cdots, n+1
\end{equation*}
hold.

Now we consider the solvability condition for equation (\ref{e2:31}) with $i = n+1$,
\begin{equation}\label{e2:35}
\int_{\mathbb{R}^n}\mu_j^{\frac{n+2}{2}}S[\tilde z](y, t)Z_{n+1}(y)dy = 0.
\end{equation}
We claim that if one choose $\mu_{0j} = b_j\mu_0(t)$ for some positive constants $b_j$, $j = 1, \cdots, k$ to be determined later, $\mu_0(t) = \gamma_{n} t^{-\frac{1}{n-2}}$ and $\gamma_{n}$ is a positive constant depending only on $n$, identity (\ref{e2:35}) holds at main order. Observe that the main contribution term to the integral on the left hand side of (\ref{e2:35}) is
\begin{equation*}
\begin{aligned}
E_{0j}&= pU(y_j)^{p-1}\left[-\mu_j^{n-3}H(q_j, q_j) + \sum_{i\neq j}\mu_j^{\frac{n-2}{2}-1}\mu_i^{\frac{n-2}{2}}G(q_j, q_i)\right]\\
&\quad +\mu_j^{-2}\dot{\mu}_jU(y_j)^{p-1}Z_{n+1}(y_j).
\end{aligned}
\end{equation*}
Then direct computations yield the following
\begin{equation*}\label{e2:36}
\begin{aligned}
&\int_{\mathbb{R}^n}\mu_j^{2}(t)E_{0j}(y, t)Z_{n+1}(y)dy\\
& \approx  c_1\left[\mu_j^{n-1}H(q_j, q_j) - \sum_{i\neq j}\mu_j^{\frac{n-2}{2}+1}\mu_i^{\frac{n-2}{2}}G(q_j, q_i)\right]+ c_2
\dot\mu_j
\end{aligned}
\end{equation*}
with
\begin{equation*}
c_1 = -p\int_{\mathbb{R}^n}U(y)^{p-1}Z_{n+1}(y)dy,
\end{equation*}
\begin{equation*}
c_2 = \int_{\mathbb{R}^n}U(y)^{p-1}\left|Z_{n+1}(y)\right|^2dy.
\end{equation*}
Note that $c_1$, $c_2$ are finite positive numbers since we assume that $n\geq 3$. Set
\begin{equation*}
\mu_j(t) = b_j\mu_0(t).
\end{equation*}
Then (\ref{e2:35}) holds at main order if we have the following identities,
\begin{equation}\label{e2:39}
b_j^{n-2}H(q_j, q_j) - \sum_{i\neq j}(b_ib_j)^{\frac{n-2}{2}}G(q_j, q_i)+c_2c_1^{-1}\mu_0^{1-n}\dot\mu_0 = 0\text{ for all } j = 1,\cdots, k.
\end{equation}
Set $c_2c_1^{-1}\mu_0^{1-n}\dot\mu_0=-\frac{2}{n-2}$, we then have
\begin{equation}\label{e2:40}
\dot{\mu}_0(t) = -\frac{2c_1c_2^{-1}}{n-2}\mu_0^{n-1}(t),
\end{equation}
with the solution $\mu_0(t) = \left(\frac{c_1^{-1}c_2}{2}\right)^{\frac{1}{n-2}}t^{-\frac{1}{n-2}}$. Furthermore, from the identities \eqref{e2:39} and \eqref{e2:40}, the constants $b_j$ must satisfy the following system
\begin{equation}\label{e2:42}
b_j^{n-3}H(q_j, q_j) - \sum_{i\neq j}b_j^{\frac{n-2}{2}-1}b_i^{\frac{n-2}{2}}G(q_j, q_i) = \frac{2}{n-2}\frac{1}{b_j}\text{ for all } j = 1,\cdots, k.
\end{equation}
System (\ref{e2:42}) can be viewed as the Euler-Lagrangian equation $\nabla_bI(b) = 0$ for the functional
\begin{equation*}\label{e2:43}
I(b): = \frac{1}{n-2}\left[\sum_{j=1}^kb_j^{n-2}H(q_j, q_j) - \sum_{i\neq j}b_j^{\frac{n-2}{2}}b_i^{\frac{n-2}{2}}G(q_j, q_i)-\sum_{j=1}^k\ln b_j^2\right].
\end{equation*}
Set $\Lambda_j = b_j^{\frac{n-2}{2}}$, then we have
\begin{equation*}
(n-2)I(b)=\tilde{I}(\Lambda)=\left[\sum_{j=1}^kH(q_j, q_j)\Lambda_j^2 - \sum_{i\neq j}G(q_j, q_i)\Lambda_i\Lambda_j-\sum_{j=1}^k\ln \Lambda_j^{\frac{4}{n-2}}\right].
\end{equation*}
By the same arguments as \cite{delPinoMussoJEMS}, system (\ref{e2:42}) possesses a unique solution with all its components be positive if and only if the matrix
\begin{eqnarray*}
\mathcal{G}(q) = \left[
\begin{matrix}
H(q_1, q_1)&-G(q_1, q_2)&\cdots & -G(q_1, q_k)\\
-G(q_2, q_1)&H(q_2, q_2)&\cdots &-G(q_2, q_k)\\
\vdots &\vdots & \ddots &\vdots\\
-G(q_k,q_1)&-G(q_k, q_2)&\cdots&H(q_k, q_k)
\end{matrix}
\right]
\end{eqnarray*}
is positive definite.
For the following solvability conditions of (\ref{e2:31}),
\begin{equation*}
\int_{\mathbb{R}^n}\mu_j^{\frac{n+2}{2}}S(\tilde z)(y, t)Z_{i}(y)dy = 0,\quad i = 1, \cdots, n,
\end{equation*}
choose $\xi_{0j} = q_j$, then these identities can be satisfied at main order. Now we denote
\begin{equation*}
\bar{\mu}_0 = (\mu_{01},\cdots, \mu_{0k}) = (b_1\mu_0,\cdots, b_k\mu_0)
\end{equation*}
and let $\Phi_j$ be the unique solution of (\ref{e2:31}) for $\mu = \bar{\mu}_0$. Then
\begin{equation*}
\Delta_y\Phi_j + pU(y)^{p-1}\Phi_j = -\mu_{0j}E_{0j}[\bar{\mu}_0,\dot{\mu}_{0j}]\text{ in }\mathbb{R}^n,\,\,\Phi_j(y, t)\to 0\text{ as }|y|\to \infty.
\end{equation*}
From the definitions of $\mu_0$ and $b_j$ as above, there holds
\begin{equation*}
\mu_{0j}E_{0j}=-\tilde\gamma_j\mu_0^{n-2}q_0(y),
\end{equation*}
where $\tilde\gamma_j$ is a positive constant and
\begin{equation*}\label{e2:50}
q_0(y): = pU(y)^{p-1}c_2 + c_1U(y)^{p-1}Z_{n+1}(y).
\end{equation*}
Let $p_0 = p_0(|y|)$ be the unique solution of $\Delta_y\Phi + pU(y)^{p-1}\Phi = q_0$, then $p_0(|y|) = O(|y|^{-2})$ as $|y|\to \infty$ and
\begin{equation*}\label{e2:51}
\Phi_j(y, t) = \tilde \gamma_j\mu_0^{n-2}p_0(y).
\end{equation*}
Now we define the improved approximation as follows
\begin{equation*}\label{e2:52}
z(x, t) = \tilde z(x, t) + \tilde{\Phi}(x, t)
\end{equation*}
with
\begin{equation*}
\tilde{\Phi}(x, t) = \sum_{j=1}^k\mu_j^{-\frac{n-2}{2}}\eta_0(x-q_j)\Phi_j\left(\frac{x-\xi_j}{\mu_j}, t\right)
\end{equation*}
and $\eta_0(x)$ is a smooth function defined on $\mathbb{R}^n$ which equals to $0$ for $x\in \mathbb{R}^n\setminus B_{\epsilon}(0)$ and equals to $1$ for $x\in B_{\frac{\epsilon}{2}}(0)$, $\epsilon > 0$ is a small but fixed positive number satisfying $0 < \epsilon < \frac{1}{2}\min\{\min_{i\neq l, i, l = 1,\cdots, k}|q_i-q_l|, \min_{i=1,\cdots, k} dist(q_i, \partial\Omega)\}$.
Here $dist(x, \partial\Omega)$ means the distance of $x$ to the boundary $\partial\Omega$ of $\Omega$. Finally we set
\begin{equation*}
\mu(t) = \bar{\mu}_0 + \lambda(t)\text{ with }\lambda(t) = (\lambda_1(t),\cdots, \lambda_k(t)).
\end{equation*}
Then the following result on the estimate of $S[z]$ holds.
\begin{lemma}\label{l2.20000}
For a fixed index $j$ and in the region $|x-q_j|\leq \frac{1}{2}\min\{\min_{i\neq l, i, l = 1,\cdots, k}|q_i-q_l|, \min_{i=1,\cdots, k} dist(q_i, \partial\Omega)\}$, $S[z]$ has the following expansion form
\begin{equation*}
\begin{aligned}
& S[z] =\sum_{j=1}^k\mu_j^{-\frac{n+2}{2}}\Bigg\{\mu_{0j}^{-1}\dot{\lambda}_jpU(y_j)^{p-1}Z_{n+1}(y_j)-2\mu_{0j}^{-2}b_j\dot\mu_0\lambda_jpU(y_j)^{p-1}Z_{n+1}(y_j)\\
&\quad\quad\quad -\mu_{0j}\mu_0^{n-4}pU(y_j)^{p-1}\sum_{i=1}^k\mathcal M_{ij}\lambda_i+\mu_j^{-2}pU(y_j)^{p-1}\dot{\xi}_j\cdot \nabla U(y_j)\\
&\quad\quad\quad +\mu_jpU(y_j)^{p-1}\Big[-\mu_j^{n-2}\nabla H(q_j, q_j) + \sum_{i\neq j}\mu_j^{\frac{n-2}{2}}\mu_i^{\frac{n-2}{2}}\nabla G(q_j, q_i)\Big]\cdot y_j\Bigg\}\\
&\quad\quad\quad + \sum_{j=1}^k\mu_j^{-\frac{n+2}{2}}\lambda_jb_j\Bigg[b_j^{-2}\mu_0^{-2}\dot{\mu}_0pU(y_j)^{p-1}Z_{n+1}(y_j)\\
&\quad\quad\quad +pU(y_j)^{p-1}\mu_0^{n-3}\bigg(-b_j^{n-4}H(q_j, q_j) + \sum_{i\neq j}b_j^{\frac{n-6}{2}}b_i^{\frac{n-2}{2}}G(q_j, q_i)\bigg)\Bigg]\\
&\quad\quad\quad +\mu_0^{-\frac{n+2}{2}}\Bigg[\sum_{j=1}^k\frac{\mu_0^{n}g_j}{1+|y_j|^{2}}+
\sum_{j=1}^k\frac{\mu_0^{2n-4}g_j}{1+|y_j|^{2}} + \sum_{j=1}^k\frac{\mu_0^{n-2}g_j}{1+|y_j|^{4}}\lambda_j\Bigg]\\
&\quad\quad\quad +\mu_0^{-\frac{n+2}{2}}\Bigg[\sum_{j=1}^k\frac{\mu_0^{n-2}\vec{g}_j}{1+|y_j|^{4}}\cdot(\xi_j-q_j)\Bigg]\\
&\quad\quad\quad +\mu_0^{-\frac{n+2}{2}}\left[\mu_0^{n-2}\sum_{i,j,l=1}^kpU(y_j)^{p-1}f_{ijl}\lambda_i\lambda_l +\sum_{i,j,l=1}^k\frac{f_{ijl}}{1+|y_j|^{n-2}}\lambda_i\dot{\lambda}_l\right]\\
&\quad\quad\quad + \mu_0^{-\frac{n+2}{2}}\left[\mu_0^{n+2}f + \mu_0^{n-1}\sum_{i=1}^k\dot{\mu}_if_i + \mu_0^{n}\sum_{i=1}^k\dot{\xi}_i\vec{f}_i\right],
\end{aligned}
\end{equation*}
where $x = \xi_j + \mu_jy_j$, $\vec{f}_i$, $f_i$, $f$, $f_{ijl}$, $g_j$ and $\vec{g}_j$ are smooth bounded functions of $(\mu_0^{-1}\mu, \xi, x)$, for $i = j$,
$$
\mathcal M_{ij} = (n-3)b_j^{n-4}H(q_j, q_j) - (\frac{n-2}{2}-1)\sum_{i\neq j}b_j^{\frac{n-2}{2}-2}b_i^{\frac{n-2}{2}}G(q_j, q_i),
$$
for $i\neq j$,
$$
\mathcal M_{ij} = - \frac{n-2}{2}\sum_{i\neq j}b_j^{\frac{n-2}{2}-1}b_i^{\frac{n-2}{2}-1}G(q_j, q_i).
$$
\end{lemma}
\noindent The proof is the same as that of \cite{delPinoMussoJEMS}, so we omit it here.
\subsection{The inner-out gluing scheme}
Now we use the ansatz
\begin{equation*}
u(x, t) = \sum_{j=1}^kz_j(x, t) + \psi(x, t)
\end{equation*}
for $z_j(x, t) = U_{\mu_j,\xi_j}(x) -  \mu_j^{\frac{n-2}{2}} H(x, q_j) + \mu_j^{-\frac{n-2}{2}}\Phi_j\left(\frac{x-\xi_j}{\mu_j}, t\right)$, with this setting, problem (\ref{e:main000}) becomes
\begin{equation*}
-\left(\left(z+\tilde{\phi}\right)^p\right)_t+\Delta \left(z+\tilde{\phi}\right)+\left(z+\tilde{\phi}\right)^p = 0,
\end{equation*}
which can be linearized as
\begin{equation}\label{e:main01}
\begin{aligned}
-pz^{p-1}\tilde\phi_t +\Delta\tilde{\phi} & + pz^{p-1}\tilde\phi + S[z] + N[\tilde{\phi}] - \left(N[\tilde{\phi}]\right)_t -\left(pz^{p-1}\right)_t\tilde\phi = 0.
\end{aligned}
\end{equation}
Here we denote
\begin{equation*}
N[\tilde{\phi}] = \left(z+\tilde{\phi}\right)^p-z^p-pz^{p-1}\tilde\phi.
\end{equation*}
Using the inner outer gluing method (see, for example, \cite{delPinoMussoJEMS} and \cite{Davila2019}), we write
\begin{equation*}
\tilde{\phi}(x, t) = \psi(x, t) + \phi^{in}(x, t)
\end{equation*}
with
\begin{equation*}
\phi^{in}(x, t) := \sum_{j=1}^k\eta_{j, R}(x, t)\tilde{\phi}_j(x, t)
\end{equation*}
\begin{equation*}
\tilde{\phi}_j(x, t) = \mu_{0j}^{-\frac{n-2}{2}}\phi\left(\frac{x-\xi_j}{\mu_{0j}}, t\right)
\end{equation*}
and
\begin{equation*}
\eta_{j, R} = \eta\left(\frac{x-\xi_j}{R\mu_{0j}}\right).
\end{equation*}
Here $\eta(s)$ is a cut-off function satisfying $\eta(s) = 1$ for $s < 1$ and $ = 0$ for $s > 2$. The positive number $R$ is independent of $t$ but sufficiently large, for convenience, we choose it as
\begin{equation}\label{e:definitionofR}
R = t_0^\varepsilon, \text{ with }0 < \varepsilon \ll 1.
\end{equation}
Then $\tilde\phi$ solves equation (\ref{e:main01}) if $\psi$ and $\tilde{\phi}^{in}$ satisfies the following system of two equations respectively
\begin{equation}\label{e:outerproblem}
\begin{cases}
\begin{aligned}
pz^{p-1}\psi_t &= \Delta\psi + V_{\mu, \xi}\psi+ \sum_{j=1}^k\left[2\nabla\eta_{j, R}\nabla_x\tilde\phi_j + \tilde\phi_j\left(\Delta_x-pU^{p-1}_j\partial_t\right)\eta_{j, R}\right]\\
&\quad + S^{*,out}_{\mu, \xi} + N[\tilde{\phi}] - \left(N[\tilde{\phi}]\right)_t -\left(pz^{p-1}\right)_t\tilde\phi\\
&\quad -pz^{p-1}\partial_t\sum_{j=1}^k\eta_{j, R}\tilde\phi_j + \sum_{j=1}^kpU^{p-1}_j\partial_t\left(\eta_{j, R}\tilde\phi_j\right)\text{ in }\Omega\times [t_0, +\infty),\\
\psi & = 0\quad \text{ on }\quad \partial\Omega\times [t_0, +\infty)
\end{aligned}
\end{cases}
\end{equation}
and
\begin{equation}\label{e:innerproblem}
\begin{aligned}
pU^{p-1}_j\partial_t\tilde\phi_j = \Delta\tilde{\phi}_j + pU^{p-1}_0\tilde\phi_j + pU^{p-1}_0\psi + S^{*,in}_{\mu, \xi, j}\quad \text{ in }B_{2R\mu_0}(\xi)\times [t_0,+\infty).
\end{aligned}
\end{equation}
Here
\begin{equation*}
V_{\mu, \xi} = \sum_{j=1}^kp\left(z^{p-1}-\left(\mu^{-\frac{n-2}{2}}_jU\left(\frac{x-\xi_j}{\mu_j}\right)\right)^{p-1}\right)\eta_{j, R} + p\left(1-\sum_{j=1}^k\eta_{j, R}\right)z^{p-1},
\end{equation*}
\begin{equation*}
U_j := \mu^{-\frac{n-2}{2}}_j U\left(\frac{x-\xi_j}{\mu_j}\right),
\end{equation*}
\begin{equation*}
\begin{aligned}
S^{*,in}_{\mu, \xi, j}(y, t) & = \mu_j^{-\frac{n+2}{2}}\Bigg\{\mu_{0j}^{-1}\dot{\lambda}_jpU(y)^{p-1}Z_{n+1}(y)-2\mu_{0j}^{-2}b_j\dot\mu_0\lambda_jpU(y)^{p-1}Z_{n+1}(y)\\
&\quad - \mu_{0j}\mu_0^{n-4}pU(y)^{p-1}\sum_{i=1}^k\mathcal M_{ij}\lambda_i+\mu_j^{-2}pU(y)^{p-1}\dot{\xi}_j\cdot \nabla U(y)\\
&\quad + \mu_jpU(y)^{p-1}\Big[-\mu_j^{n-2}\nabla H(q_j, q_j) + \sum_{i\neq j}\mu_j^{\frac{n-2}{2}}\mu_i^{\frac{n-2}{2}}\nabla G(q_j, q_i)\Big]\cdot y\Bigg\}\\
&\quad + \mu_j^{-\frac{n+2}{2}}\lambda_jb_j\Bigg[b_j^{-2}\mu_0^{-2}\dot{\mu}_0pU(y)^{p-1}Z_{n+1}(y)\\
&\quad + pU(y)^{p-1}\mu_0^{n-3}\bigg(-b_j^{n-4}H(q_j, q_j) + \sum_{i\neq j}b_j^{\frac{n-6}{2}}b_i^{\frac{n-2}{2}}G(q_j, q_i)\bigg)\Bigg]
\end{aligned}
\end{equation*}
and
\begin{equation*}
S^{*,out}_{\mu, \xi} = \left(S[z] - \sum_{j=1}^kS^{*,in}_{\mu, \xi, j}\right) + \sum_{j=1}^k(1-\eta_{j, R})S^{*,in}_{\mu, \xi, j}.
\end{equation*}
Under the self-similar coordinates, equation (\ref{e:innerproblem}) can be rewritten as
\begin{equation}\label{e:innerproblemselfsimilar}
\begin{aligned}
pU^{p-1}\partial_t\phi_j & = \Delta\phi + pU^{p-1}\phi_j + B^1[\phi_j] + B^2[\phi_j] + B^3[\phi_j] \\
&\quad + p\mu_{0j}^{\frac{n-2}{2}}\frac{\mu_{0j}^2}{\mu^2_j}U^{p-1}\left(\frac{\mu_{0j}}{\mu_j}y\right)\psi(\xi_j + \mu_{0j} y, t) + \mu_{0j}^{\frac{n+2}{2}}S^{*,in}_{\mu, \xi, j}(\xi_j + \mu_{0j} y, t)\\
&\quad \text{ in }\quad B_{2R}(0)\times [t_0,+\infty).
\end{aligned}
\end{equation}
Here
\begin{equation*}
B^1[\phi_j] = pU^{p-1}\partial_t\phi_j - p\frac{\mu_{0j}^2}{\mu^2_j}U^{p-1}\left(\frac{\mu_{0j}}{\mu_j}y\right)\partial_t\phi_j,
\end{equation*}
\begin{equation*}
B^2[\phi_j] = \mu_{0j}\dot\mu_{0j}\left(\frac{n-2}{2}\phi_j + y\cdot\nabla_y\phi_j\right) + \mu_{0j}\nabla\phi_j\cdot \dot\xi_j,
\end{equation*}
\begin{equation*}
B^3[\phi_j] = p\left[U^{p-1}\left(\frac{\mu_{0j}}{\mu_j}y\right) - U^{p-1}(y)\right]\phi_j + p\left[\frac{\mu_{0j}^2}{\mu_j^2}-1\right]U^{p-1}\left(\frac{\mu_{0j}}{\mu_j}y\right)\phi_j.
\end{equation*}
(\ref{e:outerproblem}) is the so-called outer problem, (\ref{e:innerproblemselfsimilar}) or (\ref{e:innerproblem}) is the inner problem. In Section 3, we solve the outer problem (\ref{e:outerproblem}) as a function of $\lambda$, $\xi$ and $\phi$. In Section 4, we solve the inner problem (\ref{e:innerproblem}) based on a linear theory and suitably choose of the parameter functions $\lambda$, $\xi$.
\section{The outer problem (\ref{e:outerproblem})}
\subsection{Linear theory for (\ref{e:outerproblem})}
In this subsection, we consider the linear equation of the outer problem
\begin{equation}\label{e:outerproblemmodelonomega}
\begin{cases}
\begin{aligned}
&pz^{p-1}\psi_t = \Delta\psi + V_{\mu, \xi}\psi + f(x, t) \text{ in } \Omega\times [t_0, +\infty),\\
&\psi(x, t) = 0 \text{ on }\partial\Omega\times [t_0, +\infty),\\
&\psi(x, t_0) = h(x) \text{ on } \Omega,
\end{aligned}
\end{cases}
\end{equation}
First we consider the $H^2$-estimate of (\ref{e:outerproblemmodelonomega}). We have
\begin{lemma}\label{lemma3.1000}
Suppose $\|g\|_{L^2_{t_0}, \nu} < +\infty$ and $\|h\|_{L^2(\Omega)} < +\infty$, there exists a solution $\psi = \psi(x, t)$ of the following problem
\begin{equation}\label{e4.200000}
\left\{
\begin{aligned}
-pz^{p-1}\psi_t +\Delta\psi + V_{\mu, \xi}\psi + z^{p-1}g & = 0\text{ in }\Omega\times [t_0, +\infty),\\
\psi &= 0 \text{ on }\partial\Omega\times [t_0, +\infty),\\
\psi(\cdot, t_0) & = h(x) \text{ on }\Omega,
\end{aligned}
\right.
\end{equation}
furthermore, there exists a positive constant $C$ such that
\begin{equation}\label{e:apriori11100}
\|\psi\|_{H^2_{t_0}, \nu}\leq C\left(\|h\|_{L^2(\Omega)}+\|g\|_{L^2_{t_0}, \nu}\right)
\end{equation}
holds for $t_0$ sufficiently large and $\nu > 0$.
\end{lemma}
\noindent {\bf Notations}: For $\Lambda_\tau := \Omega\times [\tau, \tau+1]$ and $\nu > 0$, we define
$$
\|\psi(\cdot, \tau)\|_{L^2} = \left(\int_{\Omega}|\psi(\cdot, \tau)|^2z^{p-1}dx\right)^{\frac{1}{2}},
$$

$$
\|\psi\|_{L^2(\Lambda_\tau)} = \left(\int\int_{\Lambda_\tau}|\psi|^2z^{p-1}dxdt\right)^{\frac{1}{2}},
$$

$$
\|\psi\|_{H^1(\Lambda_\tau)} = \|\psi\|_{L^2(\Lambda_\tau)} + \|z^{-\frac{p-1}{2}}\nabla\psi\|_{L^2(\Lambda_\tau)},
$$

$$
\|\psi\|_{H^2(\Lambda_\tau)} = \|\psi_t\|_{L^2(\Lambda_\tau)} + \|z^{-\frac{p-1}{2}}\Delta\psi\|_{L^2(\Lambda_\tau)} + \|\psi\|_{H^1(\Lambda_\tau)},
$$

$$
\|\psi\|_{L^2_{t_0}, \nu} = \sup_{\tau > t_0}\mu_0^{-\nu}\|\psi\|_{L^2(\Lambda_\tau)},
$$

$$
\|\psi\|_{H^1_{t_0}, \nu} = \sup_{\tau > t_0}\mu_0^{-\nu}\|\psi\|_{H^1(\Lambda_\tau)},
$$
$$
\|\psi\|_{H^2_{t_0}, \nu} = \sup_{\tau > t_0}\mu_0^{-\nu}\|\psi\|_{H^2(\Lambda_\tau)}.
$$
For $s > t_0$, we also define
$$
\|\psi\|_{L^2_{t_0, s}, \nu} = \sup_{t_0 < \tau < s}\mu_0^{-\nu}\|\psi\|_{L^2(\Lambda_\tau)},
$$

$$
\|\psi\|_{H^1_{t_0, s}, \nu} = \sup_{t_0 < \tau < s}\mu_0^{-\nu}\|\psi\|_{H^1(\Lambda_\tau)},
$$

$$
\|\psi\|_{H^2_{t_0, s}, \nu} = \sup_{t_0 < \tau < s}\mu_0^{-\nu}\|\psi\|_{H^2(\Lambda_\tau)},
$$
\begin{proof}
First, we consider the following problem
\begin{equation}\label{e4.200}
\left\{
\begin{aligned}
-pz^{p-1}\psi_t +\Delta\psi + V_{\mu, \xi}\psi + z^{p-1}g & = 0\text{ in }\Omega\times [t_0, s),\\
\psi(\cdot, t_0) & = h(x)\text{ in }\Omega,\\
\psi &= 0 \text{ on }\partial\Omega\times [t_0, s).
\end{aligned}
\right.
\end{equation}
Multiply (\ref{e4.200}) with $\psi$ and take integration over $\Omega$, we have
\begin{equation*}
\begin{aligned}
&\frac{p}{2}\frac{d}{dt}\int_{\Omega}\psi^2z^{p-1}dx = \int_{\Omega}\left(\Delta\psi \psi + V_{\mu, \xi}\psi^2 + \frac{p(p-1)}{2}\frac{z_t}{z}\psi^2z^{p-1}+g\psi z^{p-1}\right)dx.
\end{aligned}
\end{equation*}
Integrate by parts (since we have assumed that the boundary condition is zero) and use the Cauchy-Schwarz inequality, there holds
\begin{equation*}
\begin{aligned}
&\frac{p}{2}\frac{d}{dt}\int_{\Omega}\psi^2z^{p-1}dx + \int_{\Omega}\left|\nabla\psi\right|^2 \leq  \int_{\Omega}g^2z^{p-1}dx + \int_{\Omega}\psi^2z^{p-1}dx + \mu_0^{n-2}(t)\int_{\Omega}\psi^2z^{p-1}dx.
\end{aligned}
\end{equation*}

In the above inequality, we have used the fact that $\left|\frac{z_t}{z}\right| \lesssim \mu^{n-2}_0(t)$. Indeed, this is an Aronson-B\'{e}nilan type inequality in the setting of fast diffusion equation (see, for example, \cite{DaskalopoulosandKenig}). Observe that in the domain $\Omega\setminus B_{\varepsilon}(\xi)$ away from the blow-up point (for simplicity, we assume $k=1$ and denote $\mu_j$ as $\mu$, denote $\xi_j$ as $\xi$), $ z= \tilde{z}$ and $c\frac{1}{t}\Delta \tilde z - \partial_t\Delta \tilde z = -c\frac{\mu^{-\frac{n+2}{2}}}{t}U^{\frac{n+2}{n-2}}\left(y\right) + \mu^{-\frac{n+2}{2}}\left(-\frac{n+2}{2}\frac{\dot\mu}{\mu}\right)U^{\frac{n+2}{n-2}}\left(y\right)+ \mu^{-\frac{n+2}{2}}\frac{n+2}{n-2}U^{\frac{4}{n-2}}\left(y\right)\nabla U(y)\cdot y \left(-\frac{\dot\mu}{\mu}\right)+\mu^{-\frac{n+2}{2}}\frac{n+2}{n-2}U^{\frac{4}{n-2}}\left(y\right)\nabla U(y)\cdot \left(-\frac{\dot\xi}{\mu}\right)$.
Now if we choose the constant $c > 0$ such that $-\frac{c}{t} -\frac{n+2}{2}\frac{\dot\mu}{\mu} \approx -\frac{c}{t} +\frac{n+2}{2}\mu^{n-2}_0(t) = -\frac{c}{t} +\frac{n+2}{2}\frac{1}{t} =0$, we obtain that $c\frac{1}{t}\Delta \tilde z - \partial_t\Delta \tilde z < \mu^{-\frac{n+2}{2}}\frac{n+2}{n-2}U^{\frac{4}{n-2}}\left(y\right)\nabla U(y)\cdot y \left(-\frac{\dot\mu}{\mu}\right)+\mu^{-\frac{n+2}{2}}\frac{n+2}{n-2}U^{\frac{4}{n-2}}\left(y\right)\nabla U(y)\cdot \left(-\frac{\dot\xi}{\mu}\right) < 0$. That is to say we have $\Delta \left(c\frac{1}{t}\tilde z - \partial_t\tilde z\right) < 0$ on $\Omega\setminus B_{\epsilon}(\xi)$, moreover, there hold $\frac{1}{t}\tilde z - \partial_t\tilde z = 0$ on $\partial\Omega$ as well as the estimate $c\frac{1}{t}\tilde z - \partial_t\tilde z = c\frac{\mu^{-\frac{n-2}{2}}}{t}U\left(y\right)-\mu^{-\frac{n-2}{2}}\left(-\frac{n-2}{2}\frac{\dot\mu}{\mu}\right)U\left(y\right) - \mu^{-\frac{n-2}{2}}\nabla U\left(y\right)\cdot y\left(-\frac{\dot\mu}{\mu}\right)-\mu^{-\frac{n-2}{2}}\nabla U\left(y\right)\cdot \left(-\frac{\dot\xi}{\mu}\right) - c\frac{\mu^{-\frac{n-2}{2}}}{t}\mu^{n-2}H_\mu(x, q)+\mu^{-\frac{n-2}{2}}\left(-\frac{n-2}{2}\frac{\dot\mu}{\mu}\right)\mu^{n-2}H_\mu(x, q)+ \mu^{-\frac{n-2}{2}}\mu^{n-2}\left((n-2)\frac{\dot\mu}{\mu}\right)H_\mu(x, q)  > 0$ on $\partial B_{\epsilon}(\xi)$ when $t_0$ is large enough. From this we see that $\frac{c}{t}\tilde z - \partial_t\tilde z > 0$ in $\Omega\setminus B_{\epsilon}(\xi)$, which implies $\frac{\partial_t z}{z}\leq \frac{c}{t}$ in $\Omega\setminus B_{\epsilon}(\xi)$. Similarly, $\frac{\partial_t z}{z}\geq \frac{c'}{t}$ for some $c' < 0$ and hence $\left|\frac{\partial_t z}{z}\right|\leq \frac{c''}{t}$ in $\Omega\setminus B_{\epsilon}(\xi)$ for some positive number $c'' > 0$. In the domain $B_{\epsilon}(\xi)$, the estimate $|\frac{\partial_t z}{z}|\leq c''\mu_0^{n-2}(t)$ is obvious since the main term of $\tilde z$ is $\mu^{-\frac{n-2}{2}}U\left(y\right)$.

For $\tau\in [t_0, s-1]$, we set $\eta(t) = t-\tau$, then
\begin{equation*}
\begin{aligned}
\frac{d}{dt}\left(\eta(t)\int_{\Omega}\psi^2z^{p-1}dx\right) + \eta(t)\int_{\Omega}\left|\nabla\psi\right|^2 \leq  \int_{\Omega}(\psi^2+g^2)z^{p-1}dx + \mu_0^{n-2}(t)\int_{\Omega}\psi^2z^{p-1}dx
\end{aligned}
\end{equation*}
holds for any $t\in [\tau, \tau+1]$.
Integrate this inequality on $[\tau, \tau + 1]$, we obtain
\begin{equation*}\label{e3.14}
\begin{aligned}
&\int_{\Omega}\psi^2(\cdot, \tau+1)z^{p-1}dx + \int_{\Lambda_\tau}\eta(t)\left|\nabla\psi\right|^2dx \leq  \|\psi\|_{L^2(\Lambda_\tau)}^2 + \|g\|_{L^2(\Lambda_\tau)}^2 + \frac{1}{\tau}\|\psi\|_{L^2(\Lambda_\tau)}^2.
\end{aligned}
\end{equation*}

Multiply (\ref{e4.200}) with $\psi_t$ and take integration over $\Omega$, we have
\begin{equation*}
\begin{aligned}
\int_{\Omega}\psi^2_tz^{p-1}dx + \frac{d}{dt}\int_{\Omega}&\left(\left|\nabla\psi\right|^2 - p\psi^2z^{p-1}\right)dx \\
&\leq  \int_{\Omega}(\psi^2+g^2)z^{p-1}dx + \mu_0^{n-2}(t)\int_{\Omega}\psi^2z^{p-1}dx
\end{aligned}
\end{equation*}
and
\begin{equation*}\label{e3.15}
\begin{aligned}
\int_{\Omega}&\eta(t)\psi^2_tz^{p-1}dx + \int_{\Omega}\left(\left|\nabla\psi\right|^2 - p\psi^2z^{p-1}\right)(\cdot, \tau+1)dx\\
&\quad \leq  \|\psi\|_{L^2(\Lambda_\tau)}^2 + \|g\|_{L^2(\Lambda_\tau)}^2 +\mu_0^{n-2}(t)\|\psi\|_{L^2(\Lambda_\tau)}^2.
\end{aligned}
\end{equation*}
Therefore, we have
\begin{equation*}
\|\psi\|_{L^2_{t_0, s}, \nu}\leq  \|\psi\|_{L^2_{t_0, s}, \nu} + \|g\|_{L^2_{t_0, s}, \nu},
\end{equation*}
\begin{equation*}
\|\psi_t\|_{L^2_{t_0, s}, \nu}\leq  \|\psi\|_{L^2_{t_0, s}, \nu} + \|g\|_{L^2_{t_0, s}, \nu}
\end{equation*}
and
\begin{equation*}
\|z^{-(p-1)}\Delta\psi \|_{L^2_{t_0, s}, \nu}\leq  \|\psi\|_{L^2_{t_0, s}}^\nu + \|g\|_{L^2_{t_0, s}, \nu}.
\end{equation*}
The above estimates implies that $\|\psi\|_{H^2_{t_0, s}, \nu}\leq C(\|\psi\|_{L^2_{t_0, s}, \nu} + \|g\|_{L^2_{t_0, s}, \nu})$. Since $\int_{\Omega}V_{\mu, \xi}\psi^2dx\leq o\left(\frac{1}{R}\right)\int_{\Omega}\psi^2z^{p-1}dx$, then standard parabolic estimate shows that
$$
\|\psi\|_{L^2_{t_0, s}, \nu}\leq C \left(\|g\|_{L^2_{t_0, s}, \nu}+\|h\|_{L^2_{t_0, s}, \nu}\right).
$$
Thus we have
$$
\|\psi\|_{H^2_{t_0, s}, \nu}\leq C \left(\|g\|_{L^2_{t_0, s}, \nu}+\|h\|_{L^2_{t_0, s}, \nu}\right).
$$

Second, we consider the solution $\psi^{R, s}(x, t)$ of the following problem
\begin{equation}\label{e4.2005000}
\left\{
\begin{aligned}
pz^{p-1}\psi_t & = \Delta\psi + V_{\mu, \xi}\psi + z^{p-1}g\text{ in }Q_{R,s},\\
\psi(\cdot, t_0) & = h(x)\text{ in }\Omega_{\frac{1}{R}},\\
\psi(x, t) &= 0 \text{ on }\partial\Omega_{\frac{1}{R}}\times [t_0, s).
\end{aligned}
\right.
\end{equation}
where $Q_{R,s} = \Omega_{\frac{1}{R}}\times [t_0, s]$ and $\Omega_{\frac{1}{R}}:= \{x\in \Omega\,|\,dist(x, \partial\Omega) < \frac{1}{R}\}$, $dist(x, \partial\Omega)$ means the distance of $x$ to the boundary $\partial\Omega$ of $\Omega$. Problem (\ref{e4.2005000}) is a non-degenerate parabolic one, from standard parabolic theory, there exists a unique solution of (\ref{e4.2005000}). Then by the same arguments as above, we have
$$
\|\psi^{R, s}\|_{H^2_{t_0, s}, \nu}\leq C_0 \left(\|g\|_{L^2_{t_0, s}, \nu}+\|h\|_{L^2_{t_0, s}, \nu}\right).
$$
Here $C_0$ is independent of $R$ and $s$. Let $R_j\to +\infty$ and set $\Lambda_{\tau_0, s} = \Omega\times [t_0, s]$, then $\psi^{R_j, s}$ converges in $C^\infty(\Lambda_{\tau_0, s})$ to a smooth solution $\psi^s$ on $\Lambda_{\tau_0, s}$.

Finally, we take a sequence $s_j\to +\infty$, for each $s_j$, there exists solution $\psi^{s_j}$ satisfying the a priori estimates (\ref{e:apriori11100}) independent of $s_j$. For every compact subset $K\subset \Omega\times (t_0, +\infty)$, standard parabolic theory can be applied to get higher order derivative estimates for $\psi^s$, then, by the Arzela-Ascoli theorem, $\psi^{s_j}$ converges to a smooth solution $\psi$ of (\ref{e4.200000}) defined on $\Omega\times (t_0, +\infty)$. By taking limits, we know that estimate (\ref{e:apriori11100}) also hold, which completes the proof.
\end{proof}

In the region $\cup_{j=1}^k B_{2\mu_j R}(\xi_j)$, we consider the following model problem of (\ref{e:outerproblem}),
\begin{equation}\label{e:outerproblemmodel}
\begin{cases}
\begin{aligned}
&pz^{p-1}\psi_t = \Delta\psi + V_{\mu, \xi}\psi + f_j(x, t) \text{ in } B_{2\mu_j R}(\xi_j)\times [t_0, +\infty),\\
&\psi(\cdot, t_0) = h_j(x) \text{ on } B_{2\mu_j R}(\xi_j),
\end{aligned}
\end{cases}
\end{equation}
$j = 1,\cdots, k$. For $\alpha, \beta>0$, we assume $f_j(x, t)$ satisfies
\begin{equation}\label{assumptiononf}
|f_j(x, t)|\leq M\frac{\mu_0^{-2}\mu_0^{\beta}}{1+|y|^{2+\alpha}}
\end{equation}
and denote by $\|f_j\|_{\ast, \beta, 2+\alpha}$ the least $M$ such that (\ref{assumptiononf}) holds. It is convenient to lift (\ref{e:outerproblemmodel}) onto the standard sphere $\mathbb{S}^n$. Let us recall some facts about the conformal Laplacian on $\mathbb{S}^n$ first.

\noindent{\bf Conformal Laplacian on $\mathbb{S}^n$}. Let $\pi:\mathbb{R}^n\to \mathbb{S}^n$ be the stereographic projection given by
\begin{equation*}
\pi(y_1,\cdots, y_n) = \left(\frac{2y}{1+|y|^2}, \frac{|y|^2-1}{|y|^2+1}\right).
\end{equation*}
For a function $\phi:\mathbb{R}^n\to \mathbb{R}$, we define the lifted function $\tilde\phi$ of $\phi$ on $\mathbb{S}^n$ by the relation
\begin{equation}\label{e:lifting}
\phi(y) = \tilde\phi(\pi(y))\left(\frac{2}{1+|y|^2}\right)^{\frac{n-2}{2}},\quad y\in \mathbb{R}^n.
\end{equation}
The conformal Laplacian on $\mathbb{S}^n$ can be defined as
$$
P = \Delta_{\mathbb{S}^n} - \frac{1}{4}n(n-2),
$$
here $\Delta_{\mathbb{S}^n}$ is the Laplace-Beltrami operator on $\mathbb{S}^n$. Then the following well known property holds,
$$
\left(\frac{2}{1+|y|^2}\right)^\frac{n+2}{2} P(\tilde\phi)\circ \pi = \Delta_{\mathbb{R}^n}\phi
$$
for $\phi$ and $\tilde\phi$ satisfying the relation (\ref{e:lifting}).
Using idea of \cite{delPinoMussoJEMS}, we have the following result.
\begin{lemma}\label{lemma:lineartheoryforouterproblem}
Suppose $\|f_j\|_{*,\beta,2+\alpha}<+\infty$ for some $\alpha > 0$ and $\beta > 0$. Then there exists a solution $\psi = \psi[f_j, h_j]$ of (\ref{e:outerproblemmodel}) satisfies the following estimates
\begin{equation*}\label{e:outerestimate111}
\begin{aligned}
|\psi(x, t)|&\lesssim \|f_j\|_{*,\beta,2+\alpha}\sum_{j=1}^k\frac{\mu_0^{\beta}(t)}{1+|y_j|^{\alpha}}\\
&\quad +\sum_{j=1}^ke^{-\delta(t-t_0)}\|h_j(x)\|_{L^\infty(B_{\mu_j R}(\xi_j))},
\end{aligned}
\end{equation*}
\begin{equation*}\label{e:outerestimate11111}
\begin{aligned}
|\partial_t\psi(x, t)|&\lesssim \|f_j\|_{*,\beta,2+\alpha}\sum_{j=1}^k\frac{\mu_0^{\beta}(t)}{1+|y_j|^{\alpha-2}}
\end{aligned}
\end{equation*}
and
\begin{equation*}\label{e:outerestimate1111111}
\begin{aligned}
|\nabla\psi(x, t)|&\lesssim \|f_j\|_{*,\beta,2+\alpha}\sum_{j=1}^k\frac{\mu_0^{-1+\beta}(t)}{1+|y_j|^{\alpha-1}},
\end{aligned}
\end{equation*}
here $y_j : = \frac{x-\xi_j}{\mu_j}$.
\end{lemma}
\begin{proof}
Now we lift (\ref{e:outerproblemmodel}) to the sphere, we get the following equation
\begin{equation}\label{e:outerproblemmodelsphere}
\begin{cases}
\begin{aligned}
&(1+a(\tilde y, t))\tilde{\psi}_t = \Delta_{\mathbb{S}^n}\tilde{\psi} - \frac{1}{4}n(n-2)\tilde{\psi} + \tilde V_{\mu, \xi}\tilde\psi + \tilde{f}(\tilde{y}, t)\text{ in } \tilde B_{2R}\times [t_0, +\infty),\\
&\psi(\cdot, t_0) = \tilde{h}(\tilde y)\text{ on }\tilde B_{2R}.
\end{aligned}
\end{cases}
\end{equation}
Here $\psi(y) = \tilde\psi(\tilde y)\left(\frac{2}{1+|y|^2}\right)^{\frac{n-2}{2}}$, $\tilde y = \pi(y)$, $y =\frac{x-\xi_j}{\mu_j}\in B_{2R}(0)$ and $\tilde B_{2R}: =  \pi(B_{2R}(0))$, the functions $\tilde f$, $\tilde g$ and $\tilde h$ are defined similarly, furthermore $\tilde V_{\mu, \xi}(\tilde y, t) = \mu_j^2(1+|y|^2)^2V_{\mu, \xi}(y, t)$, $|a(\tilde y, t)| < \epsilon$ for a small number $\epsilon>0$. Note that the function $\tilde{f}(\tilde{y}, t)$ satisfies the estimate
$$
|\tilde{f}(\tilde{y}, t)|\lesssim \|f\|_{\ast, \beta, 2+\alpha}\mu_0^{\beta}(\pi-|\tilde{y}|)^{\alpha-n}.
$$
Here $|\tilde y|$ means the geodesic distance of the point $\tilde y$ to the south pole in $\mathbb{S}^n$.
Let $\tilde \psi_1$ be the solution of the following equation
\begin{equation}\label{hh11}
\left\{
\begin{aligned}
(1+a(\tilde y, t))\partial_t\tilde\psi &= \Delta_{\mathbb{S}^n}\tilde{\psi} - \frac{1}{4}n(n-2)\tilde{\psi}   \text{ in } \tilde B_{2R} \times [t_0,\infty),\\
\tilde\psi(\cdot, t_0) &= \tilde h \text{ in }\tilde B_{2R}.
\end{aligned}
\right.
\end{equation}
Suppose $\tilde v(\tilde y)$ is the bounded solution of $\Delta_{\mathbb{S}^n}\tilde v - \frac{1}{4}n(n-2)\tilde v + 1 =0$ in $\tilde B_{2R}$ satisfying $\tilde v =1$ on $\partial\tilde B_{2R}$. Then $\tilde v\ge 1$ in $\tilde B_{2R}$ and the function
\begin{equation*}
\begin{aligned}
&\bar\psi (\tilde y,t) = e^{-\delta (t-t_0)}\|\tilde h\|_{L^\infty(\tilde B_{2R})}\tilde v(\tilde y)
\end{aligned}
\end{equation*}
is a super-solution of (\ref{hh11}). Hence $|\psi_1(\tilde y, t)|\leq \bar\psi$.

Now suppose $\tilde \psi_2(\tilde y, t)$ is the unique solution of (\ref{e:outerproblemmodelsphere}) with $\tilde h = 0$. Let $p(\tilde y)$ be the positive solution of the equation
\begin{equation*}
\Delta_{\mathbb{S}^n}p - \frac{1}{4}n(n-2) p + 4q = 0\text{ in }\mathbb{S}^n
\end{equation*}
with $q(\tilde y) = \frac{1}{(\pi-|\tilde y|)^{n-\alpha}}$. Then by Riesz kernel (see, for example, \cite{Davies1989}), we get $p(\tilde y)\sim \frac{1}{(\pi-|\tilde y|)^{n-\alpha-2}}$. For a fixed small $\delta > 0$, we have
\begin{equation*}
\Delta_{\mathbb{S}^n}p - \frac{1}{4}n(n-2) p + \delta(\pi-|\tilde y|)^{-2}p+2q \leq 0\text{ in }\mathbb{S}^n.
\end{equation*}
Observe that $|\tilde V_{\mu, \xi}|\leq \delta(\pi-|\tilde y|)^{-2}$, then it is easy to see that $\tilde{\psi}(\tilde y, t) = 2\mu_0^{\beta} p$ is a positive super-solution to
\begin{equation*}
(1+a(\tilde y, t))\partial_t\tilde{\psi} = \Delta_{\mathbb{S}^n}\tilde\psi - \frac{1}{4}n(n-2)\tilde\psi + \tilde V_{\mu, \xi}\tilde{\psi} + \mu_0^{\beta}q
\end{equation*}
for $t > t_0$ and $t_0$ is large enough. Therefore, one has
\begin{equation*}\label{e4:13}
|\tilde\psi_2(\tilde y, t)|\lesssim \mu_0^{\beta}\|f\|_{*, \beta, 2+\alpha}\frac{1}{(\pi-|\tilde y|)^{n-\alpha-2}}.
\end{equation*}
Hence $\tilde \psi = \tilde \psi_1 + \tilde \psi_2$ satisfies the estimate
\begin{equation*}\label{e4:4000}
\begin{aligned}
|\tilde \psi(\tilde{y}, t)|&\lesssim \|f\|_{*,\beta,2+\alpha}\mu_0^{\beta}(t)\frac{1}{(\pi-|\tilde y|)^{n-\alpha-2}}\\
&\quad + t^{-\gamma}\|\tau^\gamma \tilde g(\tilde{y}, \tau)\|_{L^\infty(\partial\tilde B_{2R}\times [t_0,\infty))} + e^{-\delta (t-t_0)}\|\tilde h\|_{L^\infty(\tilde B_{2R})}.
\end{aligned}
\end{equation*}

Finally, scaling arguments imply that
\begin{equation*}
\begin{aligned}
|\partial_t\tilde \psi(\tilde{y}, t)|&\lesssim \|f\|_{*,\beta,2+\alpha}\mu_0^{\beta}(t)\frac{1}{(\pi-|\tilde y|)^{n-\alpha}}
\end{aligned}
\end{equation*}
and
\begin{equation*}
\begin{aligned}
|\nabla\tilde \psi(\tilde{y}, t)|&\lesssim \|f\|_{*,\beta,2+\alpha}\mu_0^{\beta}(t)\frac{1}{(\pi-|\tilde y|)^{n-\alpha-1}}\text{ for }\tilde y\in \tilde B_{2R}.
\end{aligned}
\end{equation*}
Projected to $\mathbb{R}^n$, we obtain the desired estimates.
\end{proof}
Combine the above discussions, we have the following linear theory for the outer problem. Define the norm $\|\psi\|_{**, \beta, \alpha, \nu}$ of $\psi$ as the least positive number such that
\begin{equation*}
\begin{aligned}
(1+|y|)^{-1}\mu_0|\nabla\psi(x, t)|\chi_{\cup_{j=1}^kB_{2R\mu_j}(\xi_j)}& +(1+|y|)^{-2}|\partial_t\psi(x, t)|\chi_{\cup_{j=1}^kB_{2R\mu_j}(\xi_j)}\\
&\quad +|\psi(x, t)|\chi_{\cup_{j=1}^kB_{2R\mu_j}(\xi_j)}\lesssim M\sum_{j=1}^k\frac{\mu_0^{\beta}(t)}{1+|y_j|^{\alpha}}
\end{aligned}
\end{equation*}
and
\begin{equation*}
\begin{aligned}
\|\psi\|_{H^2_{t_0}, \nu} \lesssim M.
\end{aligned}
\end{equation*}
Also we define $\|f\|_{*, \beta, 2+\alpha, \nu} = \|f\chi_{\cup_{j=1}^kB_{2R\mu_j}(\xi_j)}\|_{*, \beta, 2+\alpha} + \|z^{1-p}f\|_{L^2_{t_0}, \nu}$.
From Lemma \ref{lemma3.1000} and Lemma \ref{lemma:lineartheoryforouterproblem}, we have the following result.
\begin{prop}\label{prop3000}
There exists a bounded linear operator which maps functions $f:\Omega\times (t_0,+\infty)\to\mathbb{R}$, $h:\Omega\to\mathbb{R}$ with $\|f\|_{*, \beta, 2+\alpha, \nu}<\infty$, $\|h\|_{L^2_{t_0}, \nu} < +\infty$ into a solution $\psi$ of(\ref{e:outerproblemmodelonomega}), furthermore, the following estimate holds
\begin{equation*}
\|\psi\|_{**, \beta, \alpha, \nu}\leq C\left(\|f\|_{*, \beta, 2+\alpha, \nu} + \|h\|_{L^2(\Omega)} + e^{-\delta(t-t_0)}\|h\chi_{\cup_{j=1}^kB_{2R\mu_j}(\xi_j)}\|_{L^\infty(\Omega)}\right)
\end{equation*}
for a small constant $\delta > 0$.
\end{prop}
\subsection{Solving the outer problem (\ref{e:outerproblem})}
Given a function $h(t):(t_0, \infty)\to\mathbb{R}^k$ and $\delta > 0$, we define its weighted $L^\infty$ norm as follows
\begin{equation*}\label{e4:30}
\|h\|_\delta:=\|\mu_0(t)^{-\delta}h(t)\|_{L^\infty(t_0, \infty)}.
\end{equation*}
In the rest of this paper, we assume the parameter functions $\lambda$, $\xi$, $\dot{\lambda}$, $\dot{\xi}$ satisfy the following conditions,
\begin{equation}\label{e4:210}
\|\dot{\lambda}(t)\|_{n-1+\sigma} + \|\dot{\xi}(t)\|_{n-1+\sigma}\leq c,
\end{equation}
\begin{equation}\label{e4:220}
\|\lambda(t)\|_{1+\sigma} + \|\xi(t)-q\|_{1+\sigma}\leq c,
\end{equation}
for a positive constant $c$ which is independent of $t$, $t_0$ and $R$, $\sigma > 0$ is a small but fixed constant.
Also, for a fixed number $a\in (-n, -2)$, let us denote
\begin{equation*}\label{e4:23}
\|\phi\|_{n-2+\sigma,n+a}=\max_{j=1,\cdots, k}\|\phi_j\|_{n-2+\sigma,n+a},
\end{equation*}
where $\|\phi_j\|_{n-2+\sigma,n+a}$ is defined to be the least number $M > 0$ such that
\begin{equation}\label{e4:24}
(1+|y|)^{-2}|\partial_t \phi_j(y, t)| + (1+|y|)^{-1}|\nabla_y \phi_j(y, t)| + |\phi_j(y, t)|\leq M\frac{\mu_0^{n-2+\sigma}}{1+|y|^{n+a}}
\end{equation}
holds for $j= 1,\cdots,k$ and $|y|\leq 2R$. We assume that for $\phi = (\phi_1,\cdots, \phi_k)$, it holds that
\begin{equation}\label{e4:25}
\|\phi\|_{n-2+\sigma,n+a}\leq ct_0^{-\varepsilon}
\end{equation}
for some $\varepsilon > 0$ sufficiently small.

Note that the function $\psi$ is a solution to (\ref{e:outerproblem}) if $\psi$ is a fixed point of the operator
\begin{equation*}\label{e4:34}
\mathcal{A}(\psi):=T(f(\psi), \psi_0),
\end{equation*}
where
\begin{equation}\label{e4:35}
\begin{aligned}
f(\psi)& = \sum_{j=1}^k\left[2\nabla\eta_{j, R}\nabla_x\tilde\phi_j + \tilde\phi_j\left(\Delta_x-pU^{p-1}_j\partial_t\right)\eta_{j, R}\right]\\
&\quad + S^{*,out}_{\mu, \xi} + N[\tilde{\phi}] - \left(N[\tilde{\phi}]\right)_t -\left(pz^{p-1}\right)_t\tilde\phi\\
&\quad -pz^{p-1}\partial_t\sum_{j=1}^k\eta_{j, R}\tilde\phi_j + \sum_{j=1}^kpU^{p-1}_j\partial_t\left(\eta_{j, R}\tilde\phi_j\right).
\end{aligned}
\end{equation}
To apply the Contraction Mapping Theorem, we estimate the terms in (\ref{e4:35}) as follows:
\begin{itemize}
\item[(1)] Estimation of $S^{*,out}_{\mu, \xi}$:
\begin{equation}\label{e4:37}
\begin{aligned}
&|S^{*,out}_{\mu, \xi}(x, t)|\lesssim \mu_0^{2-\alpha-\sigma}(t_0)\sum_{j=1}^k\frac{\mu_j^{-2}\mu_0^{\frac{n-2}{2}+\sigma}}{1+|y_j|^{2+\alpha}}\\
&\text{ and }\|z^{1-p}S^{*, out}_{\mu, \xi}\|_{L^2_{t_0}, \nu} \lesssim t_0^{-\varepsilon}
\end{aligned}
\end{equation}
 with $\nu = \frac{n-2+\sigma}{2}$.
\item[(2)] Estimation of $\sum_{j=1}^k\left[2\nabla\eta_{j, R}\nabla_x\tilde\phi_j + \tilde\phi_j\left(\Delta_x-pz^{p-1}\partial_t\right)\eta_{j, R}\right]$:
\begin{equation}\label{e4:38}
\begin{aligned}
&\left|\sum_{j=1}^k\left[2\nabla\eta_{j, R}\nabla_x\tilde\phi_j + \tilde\phi_j\left(\Delta_x-pz^{p-1}\partial_t\right)\eta_{j, R}\right]\right|\lesssim \|\phi\|_{n-2+\sigma, n+a}\sum_{j=1}^k\frac{\mu_j^{-2}\mu_0^{\frac{n-2}{2}+\sigma}}{1+|y_j|^{2+\alpha}}\\
&\text{ and } \|z^{1-p}\sum_{j=1}^k\left[2\nabla\eta_{j, R}\nabla_x\tilde\phi_j + \tilde\phi_j\left(\Delta_x-pz^{p-1}\partial_t\right)\eta_{j, R}\right]\|_{L^2_{t_0}, \nu}\leq \|\phi\|_{n-2+\sigma, n+a} \end{aligned}
\end{equation}
 with $\nu = \frac{n-2+\sigma}{2}$.
\item[(3)] Estimation of $(1-\partial_t)N(\tilde{\phi})$:
\begin{equation}\label{e4:39}
\begin{aligned}
&\left|(1-\partial_t)N(\tilde{\phi})\right|\lesssim\\
&\left\{
\begin{aligned}
    t_0^{-\varepsilon}(\|\phi\|^2_{n-2+\sigma,n+a}+\|\psi\|^2_{**,\beta,\alpha})\sum_{j=1}^k\frac{\mu_j^{-2}\mu_0^{\frac{n-2}{2}+\sigma}(t)}{1+|y_j|^{2+\alpha}},  & \quad \text{when } 6 \geq n,\\
    t_0^{-\varepsilon}(\|\phi\|^p_{n-2+\sigma,n+a}+\|\psi\|^p_{**,\beta,\alpha})\sum_{j=1}^k\frac{\mu_j^{-2}\mu_0^{\frac{n-2}{2}+\sigma}(t)}{1+|y_j|^{2+\alpha}},       & \quad \text{when } 6 < n\\
  \end{aligned}
\right.\\
&\text{ and }\|z^{1-p}(1-\partial_t)N(\tilde{\phi})\|_{L^2_{t_0}, \nu}\leq c \|\psi\|_{**, \beta, \alpha, \nu}\text{ with }\nu = \frac{n-2+\sigma}{2}.
\end{aligned}
\end{equation}
\item[(4)] Estimation of $\left(pz^{p-1}\right)_t\tilde\phi$:
\begin{equation}\label{e4:3900}
\begin{aligned}
&\left|\left(pz^{p-1}\right)_t\tilde\phi\right|\lesssim \mu_0^{2-\alpha-\sigma}(t_0)\sum_{j=1}^k\frac{\mu_j^{-2}\mu_0^{\frac{n-2}{2}+\sigma}}{1+|y_j|^{2+\alpha}}\\
&\text{ and }\|z^{1-p}\left(pz^{p-1}\right)_t\tilde\phi\|_{L^2_{t_0}, \nu}\lesssim\|\phi\|_{n-2+\sigma, n+a}
\end{aligned}
\end{equation}
with $\nu = \frac{n-2+\sigma}{2}$.

\item[(5)] Estimation of $pz^{p-1}\partial_t\sum_{j=1}^k\eta_{j, R}\tilde\phi_j$:
\begin{equation}\label{e4:3900001}
\begin{aligned}
&\left|pz^{p-1}\partial_t\sum_{j=1}^k\eta_{j, R}\tilde\phi_j\right|\lesssim \|\phi\|_{n-2+\sigma, n+a}\sum_{j=1}^k\frac{\mu_j^{-2}\mu_0^{\frac{n-2}{2}+\sigma}}{1+|y_j|^{2+\alpha}}\\
&\text{ and }\|z^{1-p}\partial_t\sum_{j=1}^k\eta_{j, R}\tilde\phi_j\|_{L^2_{t_0}, \nu}\lesssim\|\phi\|_{n-2+\sigma, n+a}
\end{aligned}
\end{equation}
with $\nu = \frac{n-2+\sigma}{2}$.

\item[(6)] Estimation of $\sum_{j=1}^kpU^{p-1}_j\partial_t\left(\eta_{j, R}\tilde\phi_j\right)$:
\begin{equation}\label{e4:390000111}
\begin{aligned}
&\left|\sum_{j=1}^kpU^{p-1}_j\partial_t\left(\eta_{j, R}\tilde\phi_j\right)\right|\lesssim \|\phi\|_{n-2+\sigma, n+a}\sum_{j=1}^k\frac{\mu_j^{-2}\mu_0^{\frac{n-2}{2}+\sigma}}{1+|y_j|^{2+\alpha}}\\
&\text{ and }\|z^{1-p}\sum_{j=1}^kpU^{p-1}_j\partial_t\left(\eta_{j, R}\tilde\phi_j\right)\|_{L^2_{t_0}, \nu}\lesssim\|\phi\|_{n-2+\sigma, n+a}
\end{aligned}
\end{equation}
with $\nu = \frac{n-2+\sigma}{2}$.
\end{itemize}

{\it Proof of (\ref{e4:37})}. Recall that
\begin{equation*}\label{e4:41}
S^{*,out}_{\mu, \xi} = \left(S[z] - \sum_{j=1}^kS^{*,in}_{\mu, \xi, j}\right) + \sum_{j=1}^k(1-\eta_{j, R})S^{*,in}_{\mu, \xi, j}.
\end{equation*}
In the region $|x-q_j |>\delta$ with $\delta > 0$ small, $S^{*,out}_{\mu, \xi}$ can be estimated as follows
\begin{equation*}\label{e4:42}
\begin{aligned}
|S_{out}(x, t)|\lesssim \mu_0^{\frac{n-2}{2}}(\mu_0^{2}+\mu_0^{n})& \lesssim \mu_0^{2-\alpha-\sigma}(t_0)\sum_{j=1}^k\frac{\mu_j^{-2}\mu_0^{\frac{n-2}{2}+\sigma}}{1+|y_j|^{2+\alpha}}\\
&\lesssim t_0^{-\varepsilon}\sum_{j=1}^k\frac{\mu_j^{-2}\mu_0^{\frac{n-2}{2}+\sigma}}{1+|y_j|^{2+\alpha}}.
\end{aligned}
\end{equation*}
In the region $|x-q_j| \leq \delta$ with $\delta > 0$ small, we have
\begin{equation*}\label{e4:43}
\begin{aligned}
\left|S^{(2)}_{\mu,\xi}(x, t)\right|\lesssim \mu_0^{-\frac{n+2}{2}}\frac{\mu_0^{n}}{1+|y_j|^{2}}&\lesssim  \mu_0^{2-\alpha-\sigma}(t_0)\sum_{j=1}^k\frac{\mu_j^{-2}\mu_0^{\frac{n-2}{2}+\sigma}}{1+|y_j|^{2+\alpha}}\lesssim t_0^{-\varepsilon}\sum_{j=1}^k\frac{\mu_j^{-2}\mu_0^{\frac{n-2}{2}+\sigma}}{1+|y_j|^{2+\alpha}}.
\end{aligned}
\end{equation*}
Furthermore, in the region $|x-q_j| < \delta$,
\begin{equation*}\label{e4:44}
\begin{aligned}
\left|(1 - \eta_{j, R})S_{\mu, \xi, j}^{*, in}\right|&\lesssim t_0^{-\varepsilon}\sum_{j=1}^k\frac{\mu_j^{-2}\mu_0^{\frac{n-2}{2}+\sigma}}{1+|y_j|^{2+\alpha}}
\end{aligned}
\end{equation*}
since $(1-\eta_{j, R})\neq 0$ if $|x-\xi_j|>\mu_0R$. Therefore, we have $\|S^{*,out}_{\mu, \xi}\|_{*, \beta, 2+\alpha} < t_0^{-\varepsilon}$.
Similarly, we have
\begin{equation}\label{e:l2estimate}
\begin{aligned}
\int_{\Omega}\left|z^{1-p}S^{*,out}_{\mu, \xi}\right|^2z^{p-1}dx&\leq t_0^{-\varepsilon}\int_{\Omega}\left|\frac{\mu_0^{\frac{n-2}{2}+\sigma}}{1+|y|^{\frac{7}{2}-\sigma}}|y|^4\right|^2z^{p-1}dx\\
& \leq t_0^{-\varepsilon}\int_{\Omega/\mu_0}\left|\frac{\mu_0^{n-2+\sigma}}{1+|y|^{\frac{7}{2}-\sigma}}|y|^4\right|^2\frac{1}{1+|y|^4}dy\\
& \leq t_0^{-\varepsilon}\mu_0^{n-2+\sigma}\int_{\Omega/\mu_0}\frac{1}{1+|y|^{-2\sigma-1+n-2+\sigma}}\frac{1}{1+|y|^4}dy\\
& \leq t_0^{-\varepsilon}\mu_0^{n-2+\sigma}\int_{\mathbb{R}^n}\frac{1}{1+|y|^{-\sigma+n+1}}dy\\
& \leq  t_0^{-\varepsilon}\mu_0^{n-2+\sigma},
\end{aligned}
\end{equation}
thus $\|z^{1-p}S^{*, out}_{\mu, \xi}\|_{L^2_{t_0}, \nu} \leq t_0^{-\varepsilon}$ with $\nu = \frac{n-2+\sigma}{2}$.

{\it Proof of (\ref{e4:38})}.
For the term $\tilde{\phi}_j\big(\Delta -\partial_t\big)\eta_{j, R}$, we have
\begin{equation*}\label{e4:46}
\begin{aligned}
\left|\tilde{\phi}_j\big(\Delta -\partial_t\big)\eta_{j, R}\right|\lesssim & \frac{\left|\Delta\eta\left(|\frac{x-\xi_j}{R\mu_{0j}}|\right)\right|}{R^{2}\mu_{0j}^{2}}\mu_0^{-\frac{n-2}{2}}|\phi_j|\\
&+\left|\eta'\left(|\frac{x-\xi_j}{R\mu_{0j}}|\right)\left(\frac{|x-\xi_j|}{R\mu_0^2}\dot{\mu_0}+\frac{1}{R\mu_0}\dot{\xi}\right)\right|\mu_0^{-\frac{n-2}{2}}|\phi_j|.
\end{aligned}
\end{equation*}
Furthermore, there hold
\begin{equation*}\label{e4:47}
\begin{aligned}
\frac{\left|\Delta\left(|\frac{x-\xi_j}{R\mu_{0j}}|\right)\right|}{R^{2}\mu_{0j}^{2}}\mu_0^{-\frac{n-2}{2}}|\phi_j|& \lesssim  \frac{\left|\Delta\eta\left(|\frac{x-\xi_j}{R\mu_{0j}}|\right)\right|}{R^{2}\mu_{0j}^{2}}\frac{\mu_0^{\frac{n-2}{2}+\sigma}}{(1+|y_j|^{n+a})}\|\phi\|_{n-2+\sigma, n+a}\\
&\lesssim \|\phi\|_{n-2+\sigma, n+a}\sum_{j=1}^k\frac{\mu_j^{-2}\mu_0^{\frac{n-2}{2}+\sigma}(t)}{1+|y_j|^{2+\alpha}}
\end{aligned}
\end{equation*}
and
\begin{equation*}
\begin{aligned}
&\left|\eta'\left(|\frac{x-\xi_j}{R\mu_{0j}}|\right)\left(\frac{|x-\xi_j|\dot{\mu_0}+\mu_0\dot{\xi}}{R\mu_0^2}\right)\right|\mu_0^{-\frac{n-2}{2}}|\phi_j|\\
&\quad\quad\quad\quad\quad\lesssim \frac{\left|\eta'\left(|\frac{x-\xi_j}{R\mu_{0j}}|\right)\right|}{R^{2}\mu_{0j}^{2}}(\mu_0^{n}R^{2} +\mu_0^{n+\sigma}R)\mu_0^{-\frac{n-2}{2}}|\phi_j|\\
&\quad\quad\quad\quad\quad\lesssim \|\phi\|_{n-2+\sigma, n+a}\sum_{j=1}^k\frac{\mu_j^{-2}\mu_0^{\frac{n-2}{2}+\sigma}(t)}{1+|y_j|^{2+\alpha}}.
\end{aligned}
\end{equation*}
The estimate of $\nabla\eta_{j, R}\cdot\nabla\tilde{\phi}_j - \tilde\phi_j pz^{p-1}\partial_t\eta_{j, R}$ is similar, hence we have (\ref{e4:38}). Therefore, we have $$\|\sum_{j=1}^k\left[2\nabla\eta_{j, R}\nabla_x\tilde\phi_j + \tilde\phi_j\left(\Delta_x-pz^{p-1}\partial_t\right)\eta_{j, R}\right]\|_{*, \beta, 2+\alpha} \lesssim \|\phi\|_{n-2+\sigma, n+a}.$$
Similar estimates as (\ref{e:l2estimate}), we have $$\|z^{1-p}\sum_{j=1}^k\left[2\nabla\eta_{j, R}\nabla_x\tilde\phi_j + \tilde\phi_j\left(\Delta_x-pz^{p-1}\partial_t\right)\eta_{j, R}\right]\|_{L^2_{t_0}, \nu}\lesssim \|\phi\|_{n-2+\sigma, n+a}$$ with $\nu = \frac{n-2+\sigma}{2}$.

{\it Proof of (\ref{e4:39})}.
Observe that
\begin{equation*}\label{e4:49}
\begin{aligned}
&N(\psi + \sum_{j=1}^k\eta_{j, R}\tilde{\phi}_j)\lesssim\left\{
\begin{aligned}
&z^{p-2}\left[|\psi|^2 + \sum_{j=1}^k|\eta_{j, R}\tilde{\phi}_j|^2\right], & \quad \mbox{when}~ 6\geq n,\\
&|\psi|^p + \sum_{j=1}^k|\eta_{j, R}\tilde{\phi}_j|^p, & \quad \mbox{when}~ 6 < n.
\end{aligned}
\right.
\end{aligned}
\end{equation*}
If $6\geq n$, there hold
\begin{equation*}
\begin{aligned}
\left|z^{p-2}(\eta_{j, R}\tilde{\phi}_j)^2\right|&\lesssim |\frac{\tilde\phi_j}{z}z^{p-1}\tilde\phi_j|\lesssim \mu_0^\sigma\|\phi\|^2_{n-2+\sigma, n+a}\frac{\mu_0^{\frac{n-2}{2}+\sigma}}{1+|y_j|^{4}}\\
&\lesssim t_0^{-\varepsilon}\|\phi\|^2_{n-2+\sigma, n+a}\sum_{j=1}^k\frac{\mu_j^{-2}\mu_0^{\frac{n-2}{2}+\sigma}(t)}{1+|y_j|^{2+\alpha}}
\end{aligned}
\end{equation*}
and
\begin{equation*}
\begin{aligned}
\left|z^{p-2}\psi^2\right|\lesssim |\frac{\psi}{z}z^{p-1}\psi| & \lesssim \mu_0^\sigma\|\psi\|^2_{**,\beta, \alpha}\frac{\mu_0^{\frac{n-2}{2}+\sigma}}{1+|y_j|^{4+\alpha}}\\
&\lesssim t_0^{-\varepsilon}\|\psi\|^2_{**,\beta, \alpha}\sum_{j=1}^k\frac{\mu_j^{-2}\mu_0^{\frac{n-2}{2}+\sigma}(t)}{1+|y_j|^{2+\alpha}}.
\end{aligned}
\end{equation*}
In the above, we have used the facts that $\left|\frac{\tilde\phi_j}{z}\right|\leq \mu_0^\sigma(t)\|\phi\|_{n-2+\sigma, n}$ and $\left|\frac{\psi}{z}\right|\leq \mu_0^\sigma(t)\|\psi\|_{**,\beta, \alpha}$ in the region $\cup_{j=1}^kB_{2R\mu_j}(\xi_j)$. If $6 < n$, there hold
\begin{equation*}
\begin{aligned}
\left|\eta_{j, R}\tilde{\phi}_j\right|^p&\lesssim \frac{\mu_0^{(\frac{n-2}{2}+\sigma) p}}{1+|y_j|^{(n+a)p}}\|\phi\|^p_{n-2+\sigma, n+a}\\
&\lesssim \mu_0^{(p-1)(\frac{n-2}{2}+\sigma)}\|\phi\|^p_{n-2+\sigma, n+a}\sum_{j=1}^k\frac{\mu_j^{-2}\mu_0^{\frac{n-2}{2}+\sigma}(t)}{1+|y_j|^{2+\alpha}},
\end{aligned}
\end{equation*}
and
\begin{equation*}
\begin{aligned}
\left|\psi\right|^p&\lesssim\frac{\mu_0^{p(\frac{n-2}{2}+\sigma)}}{1+|y_j|^{p\alpha}}\|\psi\|^p_{**,\beta, a}\\
&\lesssim \mu_0^{(p-1)(\frac{n-2}{2}+\sigma)}\|\psi\|^p_{**,\beta,\alpha}\sum_{j=1}^k\frac{\mu_j^{-2}\mu_0^{\frac{n-2}{2}+\sigma}(t)}{1+|y_j|^{2+\alpha}}.
\end{aligned}
\end{equation*}
The estimates for $\partial_tN$ are similar.

Since
\begin{equation*}
\begin{aligned}
\left|z^{1-p}z^{p-2}\psi^2\right|\lesssim |\frac{\psi}{z}\psi| & \lesssim c |\psi|
\end{aligned}
\end{equation*}
and
\begin{equation*}
\begin{aligned}
\left|z^{1-p}|\psi|^{p}\psi^2\right|\lesssim |\left(\frac{\psi}{z}\right)^{p-1}\psi| & \lesssim c |\psi|,
\end{aligned}
\end{equation*}
we have $$\|z^{1-p}(1-\partial_t)N(\tilde{\phi})\|_{L^2_{t_0}, \nu}\leq c \|\psi\|_{**, \beta, \alpha, \nu}$$ with $\nu = \frac{n-2+\sigma}{2}$.

Here we have used the fact that: in the region $\Omega\setminus B_{2R\mu_j}(\xi_j)$, the solution $\psi$ of (\ref{e:outerproblem}) satisfying the estimate
$$
|\psi(x, t)|\lesssim  |z(x, t)|.
$$
Indeed, observe that in the region $\Omega\setminus \cup_{j=1}^kB_{2R\mu_j}(\xi_j)$, the function $u(x, t) = z(x, t)+\psi(x, t)$ is a solution of the problem
\begin{equation}\label{e:outerproblenearboundary}
\begin{cases}
\begin{aligned}
\frac{\partial u^p}{\partial t} & = \Delta u + u^p \text{ in }\Omega\times (t_0, +\infty),\\
u & = 0\text{ on }\partial\Omega\times (t_0, +\infty),\\
u (x, t_0) &= u_0(x) : = z(x, t_0)+\psi_0(x)\text{ on }\Omega.
\end{aligned}
\end{cases}
\end{equation}
Suppose $v = v(x)$ is the bounded solution of $\Delta v + 1 =0$ in $\Omega$ satisfying $v = 0$ on $\partial \Omega$. Then $v > 0$ in $\Omega$ and the function
\begin{equation*}
\begin{aligned}
&\bar\psi (x,\tau) = \left(T-\tau\right)^{\frac{1+\delta}{1-m}}v(x)^{\frac{1}{m}}\text{ with }m = \frac{n-2}{n+2}
\end{aligned}
\end{equation*}
is a super-solution of $\partial_\tau w - \Delta w^m = 0$. Indeed, we have
\begin{equation*}
\begin{aligned}
\partial_\tau \bar\psi - \Delta \bar\psi^m &= -\frac{1+\delta}{1-m} \left(T-\tau\right)^{\frac{m+\delta}{1-m}}v(x)^{\frac{1}{m}} + \left(T-\tau\right)^{\frac{m(1+\delta)}{1-m}}\\
& = \left(T-\tau\right)^{\frac{m(1+\delta)}{1-m}}\left(-\frac{1+\delta}{1-m} \left(T-\tau\right)^{\frac{\delta-\delta m}{1-m}}v(x)^{\frac{1}{m}} + 1\right) > 0
\end{aligned}
\end{equation*}
when $\tau$ is close to $T$. Then by the maximum principal for the fast diffusion equation (for example, Theorem 1.1.1 in \cite{DaskalopoulosandKenig}), we have $|w(x, \tau)|\leq \left(T-\tau\right)^{\frac{1+\delta}{1-m}}v(x)^{\frac{1}{m}}$ when $\tau$ is close to $T$. From the relation (\ref{e:transformation}), the solution of (\ref{e:outerproblenearboundary})
can be controlled as $|u(x, t)|\leq \left(T-\tau\right)^{\frac{m\delta}{1-m}} v(x)\leq \left(Te^{-t}\right)^{\frac{m\delta}{1-m}}v(x)$ if $u_0: = z(x, t_0) + \psi_0(x)$ satisfies $\|u_0\|_{L^\infty(\Omega)}\leq e^{-\varepsilon t_0}$ for $t_0 > 0$ large enough and $\varepsilon > 0$ is small enough. Hence in the region $\Omega\setminus \cup_{j=1}^kB_{\epsilon}(\xi_j)$ with $\epsilon > 0$ small enough, the solution $\psi$ of (\ref{e:outerproblem}) satisfies the esitmate
\begin{equation}\label{e:estimateofpsi}
|\psi|\lesssim |z| + \left(T-\tau\right)^{\frac{m\delta}{1-m}} v(x) \lesssim |z|+\left(Te^{-t}\right)^{\frac{m\delta}{1-m}}v(x).
\end{equation}
Furthermore, $|z|\leq C\mu^{\frac{n-2}{2}}_0(t) v(x)$ in the $\Omega\setminus \cup_{j=1}^kB_{\epsilon}(\xi_j)$ for some positive constant $C > 0$, $\epsilon > 0$ is a fixed small number. Indeed, $z$ satisfies $\Delta z + \mu^{-\frac{n+2}{2}}U^{\frac{n+2}{n-2}}(y)=0$ in $\Omega\setminus B_{\epsilon}(\xi)$, $z = 0 $ on $\partial \Omega$, $z > C\mu^{\frac{n-2}{2}}v(x) $ on $\partial B_{\epsilon}(\xi)$ (for simplicity, we assume $k=1$ and denote $\xi_j$ as $\xi$). From this we see that $z > C\mu^{(n-2)/2}  v (x)$ in $\Omega\setminus B_{\epsilon}(\xi)$ and $\left(Te^{-t}\right)^{\frac{m\delta}{1-m}}v(x)/z\lesssim \left(Te^{-t}\right)^{\frac{m\delta}{1-m}}\mu^{-\frac{n-2}{2}}\ll 1$ when $t_0$ is large. In the region $B_{\epsilon}(\xi)$, we have $\left(Te^{-t}\right)^{\frac{m\delta}{1-m}}v(x)/z\lesssim \left(Te^{-t}\right)^{\frac{m\delta}{1-m}}\mu^{\frac{n-2}{2}-(n-2)}\ll 1$. From (\ref{e:estimateofpsi}), we obtain $|\psi(x, t)|\lesssim  |z(x, t)|$.

{\it Proof of (\ref{e4:3900})}. From the definition of $\|\phi\|_{n-2+\sigma, n+a}$, we have
\begin{equation*}
\begin{aligned}
\left|\left(pz^{p-1}\right)_t\tilde\phi\right|\lesssim \left|z^{p-1}\tilde\phi\right|\left|\frac{\dot\mu +\dot\xi}{\mu}\right|& \lesssim\mu_0^{n-2}\|\phi\|_{n-2+\sigma, n+a}\sum_{j=1}^k\frac{\mu_j^{-2}\mu_0^{\frac{n-2}{2}+\sigma}}{1+|y_j|^{n+a+4}}\\
& \lesssim t_0^{-\varepsilon}\|\phi\|_{n-2+\sigma, n+a}\sum_{j=1}^k\frac{\mu_j^{-2}\mu_0^{\frac{n-2}{2}+\sigma}}{1+|y_j|^{2+\alpha}}.
\end{aligned}
\end{equation*}
Therefore, we have $$\|\left(pz^{p-1}\right)_t\tilde\phi\|_{*, \beta, 2+\alpha} \lesssim \|\phi\|_{n-2+\sigma, n+a}.$$
Similar to (\ref{e:l2estimate}), we have $$\|z^{1-p}\left(pz^{p-1}\right)_t\tilde\phi\|_{L^2_{t_0}, \nu}\lesssim\|\phi\|_{n-2+\sigma, n+a}$$ with $\nu = \frac{n-2+\sigma}{2}$.

{\it Proof of (\ref{e4:3900001})}. From the definition of $\|\phi\|_{n-2+\sigma, n+a}$, we have
\begin{equation*}
\begin{aligned}
&\left|pz^{p-1}\partial_t\sum_{j=1}^k\eta_{j, R}\tilde\phi_j\right|\\ &\lesssim \left|pz^{p-1}\right|\sum_{j=1}^k\left(\left|\partial_t\eta_{j, R}\right|\left|\tilde\phi_j\right|+\left|\eta_{j, R}\right|\left|\partial_t\tilde\phi_j\right|\right)\\
&\lesssim\left|pz^{p-1}\right|\sum_{j=1}^k\left|\eta'\left(|\frac{x-\xi_j}{R\mu_{0j}}|\right)\left(\frac{|x-\xi_j|}{R\mu_0^2}\dot{\mu_0}+\frac{1}{R\mu_0}\dot{\xi}\right)\right|
\mu_0^{-\frac{n-2}{2}}|\phi_j|\\
&\quad +\left|pz^{p-1}\right|\sum_{j=1}^k\left|\eta_{j, R}\right|\left(\mu_0^{-\frac{n-2}{2}}|\partial_t\phi_j|+\mu_0^{-\frac{n-2}{2}}\frac{\dot\mu_0}{\mu_0}|\phi_j|\right)\\
&\lesssim \|\phi\|_{n-2+\sigma, n+a}|z^{p-1}|\frac{\mu_0^{\frac{n-2}{2}+\sigma}}{1+|y|^{n+a}}\lesssim \|\phi\|_{n-2+\sigma, n+a}\sum_{j=1}^k\frac{\mu_j^{-2}\mu_0^{\frac{n-2}{2}+\sigma}}{1+|y_j|^{2+\alpha}}.
\end{aligned}
\end{equation*}
Therefore, we have $$\|pz^{p-1}\partial_t\sum_{j=1}^k\eta_{j, R}\tilde\phi_j\|_{*, \beta, 2+\alpha} \lesssim \|\phi\|_{n-2+\sigma, n+a}.$$
Similar to (\ref{e:l2estimate}), we have $$\|p\partial_t\sum_{j=1}^k\eta_{j, R}\tilde\phi_j\|_{L^2_{t_0}, \nu}\lesssim\|\phi\|_{n-2+\sigma, n+a}$$ with $\nu = \frac{n-2+\sigma}{2}$.

{\it Proof of (\ref{e4:390000111})}. From the definition of $\|\phi\|_{n-2+\sigma, n+a}$, we have
\begin{equation*}
\begin{aligned}
&\left|\sum_{j=1}^kpU^{p-1}_j\partial_t\left(\eta_{j, R}\tilde\phi_j\right)\right|\\ &\lesssim \sum_{j=1}^k\left|pU^{p-1}_j\right|\left(\left|\partial_t\eta_{j, R}\right|\left|\tilde\phi_j\right|+\left|\eta_{j, R}\right|\left|\partial_t\tilde\phi_j\right|\right)\\
&\lesssim\sum_{j=1}^k\left|pU^{p-1}_j\right|\left|\eta'\left(|\frac{x-\xi_j}{R\mu_{0j}}|\right)\left(\frac{|x-\xi_j|}{R\mu_0^2}\dot{\mu_0}+\frac{1}{R\mu_0}\dot{\xi}\right)\right|
\mu_0^{-\frac{n-2}{2}}|\phi_j|\\
&\quad +\sum_{j=1}^k\left|pU^{p-1}_j\right|\left|\eta_{j, R}\right|\left(\mu_0^{-\frac{n-2}{2}}|\partial_t\phi_j|+\mu_0^{-\frac{n-2}{2}}\frac{\dot\mu_0}{\mu_0}|\phi_j|\right)\\
&\lesssim \|\phi\|_{n-2+\sigma, n+a}|z^{p-1}|\frac{\mu_0^{\frac{n-2}{2}+\sigma}}{1+|y|^{n+a}}\lesssim \|\phi\|_{n-2+\sigma, n+a}\sum_{j=1}^k\frac{\mu_j^{-2}\mu_0^{\frac{n-2}{2}+\sigma}}{1+|y_j|^{2+\alpha}}.
\end{aligned}
\end{equation*}
Therefore, we have $$\|\sum_{j=1}^kpU^{p-1}_j\partial_t\left(\eta_{j, R}\tilde\phi_j\right)\|_{*, \beta, 2+\alpha} \lesssim \|\phi\|_{n-2+\sigma, n+a}.$$
Similar to (\ref{e:l2estimate}), we have $$\|z^{1-p}\sum_{j=1}^kpU^{p-1}_j\partial_t\left(\eta_{j, R}\tilde\phi_j\right)\|_{L^2_{t_0}, \nu}\lesssim\|\phi\|_{n-2+\sigma, n+a}$$ with $\nu = \frac{n-2+\sigma}{2}$.

Now we set
\begin{equation*}\label{e4:54}
\mathcal{B}=\left\{\psi:\|\psi\|_{**, \beta, \alpha, \nu}\leq M t_0^{-\varepsilon}\right\}
\end{equation*}
with $\beta = \frac{n-2}{2}+\sigma$ and $\nu = \frac{n-2+\sigma}{2}$. Here the constant $M$ is large but independent of $t$ and $t_0$.
For any $\psi\in \mathcal{B}$, $\mathcal{A}(\psi)\in \mathcal{B}$ as a consequence of the estimations (\ref{e4:37})-(\ref{e4:390000111}). And similar estimations imply that, for any $\psi_1$, $\psi_2\in\mathcal{B}$, there holds
\begin{equation*}\label{e4:55}
\|\mathcal{A}(\psi^{(1)}) - \mathcal{A}(\psi^{(2)})\|_{**,\beta, \alpha, \nu}\leq C\|\psi^{(1)}-\psi^{(2)}\|_{**,\beta, \alpha, \nu},
\end{equation*}
for a constant $C<1$ when $t_0$ is chosen large enough. Therefore, $\mathcal{A}$ is a contraction map in $\mathcal{B}$ and there exists a fixed point $\psi$ of $\mathcal{A}$, which is a solution to the outer problem (\ref{e:outerproblem}).
Therefore, we obtain the following result.
\begin{prop}\label{propositionouterproblem}
Assume $\lambda$, $\xi$, $\dot{\lambda}$, $\dot{\xi}$ satisfy the conditions (\ref{e4:210}) and (\ref{e4:220}), $\phi = (\phi_1,\cdots, \phi_k)$ satisfies conditions (\ref{e4:25}), $\psi_0\in C^2(\Omega)$ and
\begin{equation*}\label{e4:26}
\|\psi_0\|_{L^{\infty}(\Omega)} + \|\psi_0\|_{L^2_{t_0}, \nu}\leq t_0^{-\varepsilon}
\end{equation*}
for $\nu = \frac{n-2+\sigma}{2}$. Then there exists a large enough $t_0 > 0$ and a small constant $\alpha > 0$ such that the outer problem (\ref{e:outerproblem}) possesses a unique solution $\psi = \Psi[\lambda, \xi, \dot{\lambda}, \dot{\xi}, \phi]$. Moreover, there hold
\begin{equation*}\label{e4:27}
|\psi(x, t)|\chi_{\cup_{j=1}^kB_{2R}(\xi_j)}\lesssim t_0^{-\varepsilon}\sum_{j=1}^k\frac{\mu_0^{\frac{n-2}{2}+\sigma}(t)}{1+|y_j|^{\alpha}} + \sum_{j=1}^ke^{-\delta(t-t_0)}\|\psi_0\|_{L^{\infty}(\Omega)},
\end{equation*}
\begin{equation*}
\begin{aligned}
|\nabla\psi(x, t)|\chi_{\cup_{j=1}^kB_{2R}(\xi_j)}&\lesssim t_0^{-\varepsilon}\sum_{j=1}^k\frac{\mu_0^{-1+\frac{n-2}{2}+\sigma}(t)}{1+|y_j|^{\alpha-1}}
\end{aligned}
\end{equation*}
and
$$
\|\psi\|_{H^2_{t_0}, \nu} \lesssim t_0^{-\varepsilon}.
$$
Here $y_j=\frac{x-\xi_j}{\mu_{0j}}$.
\end{prop}

\begin{remark}\label{propositionouterproblem1}
The solution $\Psi$ obtained in Proposition \ref{propositionouterproblem} depends smoothly on the parameters $\lambda$, $\xi$, $\dot{\lambda}$, $\dot{\xi}$, $\phi$, for $y_j=\frac{x-\xi_j}{\mu_{0j}}$. Indeed, using Lemma \ref{lemma:lineartheoryforouterproblem} and the same arguments as Proposition 4.2 of \cite{delPinoMussoJEMS}, in the domain $\cup_{j=1}^kB_{2R\mu_j}(\xi_j)$, we have
\begin{equation*}\label{e4:64}
\big|\partial_\lambda\Psi[\lambda,\xi,\dot{\lambda},\dot{\xi},\phi][\bar{\lambda}](x, t)\big|\lesssim t_0^{-\varepsilon}\|\bar{\lambda}(t)\|_{1+\sigma}\left(\sum_{j=1}^k\frac{\mu_0^{\frac{n-2}{2}+\sigma-1}(t)}{1+|y_j|^{\alpha}}\right),
\end{equation*}
\begin{equation*}\label{e4:74}
\big|\partial_\xi\Psi[\lambda,\xi,\dot{\lambda},\dot{\xi},\phi][\bar{\xi}](x, t)\big|\lesssim t_0^{-\varepsilon}\|\bar{\xi}(t)\|_{1+\sigma}\left(\sum_{j=1}^k\frac{\mu_0^{\frac{n-2}{2}+\sigma-1}(t)}{1+|y_j|^{\alpha}}\right),
\end{equation*}
\begin{equation*}\label{e4:75}
\big|\partial_{\dot{\xi}}\Psi[\lambda,\xi,\dot{\lambda},\dot{\xi},\phi][\dot{\bar{\xi}}](x, t)\big|\lesssim t_0^{-\varepsilon}\mu_0^{n-1+\sigma}\|\dot{\bar{\xi}}(t)\|_{n-1+\sigma}\left(\sum_{j=1}^k\frac{\mu_0^{-\frac{n}{2}+\sigma}(t)}{1+|y_j|^{\alpha}}\right),
\end{equation*}
\begin{equation*}\label{e4:76}
\big|\partial_{\dot{\lambda}}\Psi[\lambda,\xi,\dot{\lambda},\dot{\xi},\phi][\dot{\bar{\lambda}}](x, t)\big|\lesssim t_0^{-\varepsilon}\mu_0^{n-1+\sigma}\|\dot{\bar{\lambda}}(t)\|_{n-1+\sigma}\left(\sum_{j=1}^k\frac{\mu_0^{-\frac{n}{2}+\sigma}(t)}{1+|y_j|^{\alpha}}\right),
\end{equation*}
\begin{equation*}\label{e4:84}
\big|\partial_{\phi}\Psi[\lambda,\xi,\dot{\lambda},\dot{\xi},\phi][\bar{\phi}](x, t)\big|\lesssim \|\bar{\phi}(t)\|_{n-2+\sigma, n+a}\left(\sum_{j=1}^k\frac{\mu_0^{\frac{n-2}{2}+\sigma}(t)}{1+|y_j|^{\alpha}}\right).
\end{equation*}
\end{remark}

\section{The inner problem (\ref{e:innerproblemselfsimilar})}
To solve the highly nonlinear problem (\ref{e:innerproblemselfsimilar}), we need a linear theory first, which is the content of
\subsection{The linear theory of the inner problem (\ref{e:innerproblemselfsimilar})}
In this subsection, we consider the following linear equation
\begin{equation}\label{e:mainsection3}
\begin{aligned}
-pU^{p-1}\phi_t +\Delta\phi + pU^{p-1}\phi + U^{p-1}h = 0 \text{ on }\mathbb{R}^n,
\end{aligned}
\end{equation}
with $h = h(y, t)$ being supported on the ball $B_{2R}(0)$ and under the orthogonality conditions
\begin{equation}\label{e:orthogonalitycondition111z0}
\int_{B_{2R}}h(y, t)Z_j(y)U^{p-1}(y)dy = 0\text{ for }j = 0, 1,\cdots, n+1.
\end{equation}
Equation (\ref{e:mainsection3}) is a degenerate parabolic equation, therefore a natural way is to lift it to the standard sphere $\mathbb{S}^n$, which becomes a classical (non-degenerate) parabolic problem on $\mathbb{S}^n$. Similarly to (\ref{e:lifting}), we define $\tilde{g}$ on $\mathbb{S}^n$ to be
\begin{equation*}
h(y) = \tilde h(\pi(y))\left(\frac{2}{1+|y|^2}\right)^{\frac{n-2}{2}},\quad y\in \mathbb{R}^n.
\end{equation*}
Then standard computation shows that (\ref{e:mainsection3}) is equivalent to the following linear heat problem on $\mathbb{S}^n$
\begin{equation}\label{e:linearproblemonsphere}
\partial_t\tilde\phi = \left(\Delta_{\mathbb{S}^n} + \lambda_1\right)\tilde\phi + \tilde h \quad\text{ on }\quad \mathbb{S}^n.
\end{equation}
Here $\lambda_1 = n$ is the second eigenvalue of $\Delta_{\mathbb{S}^n}$ with eigenfunctions $\tilde Z_j$, $j = 1,\cdots, n+1$, given by the functions
\begin{equation*}
Z_i(y) = \tilde Z_i(\pi(y))\left(\frac{2}{1+|y|^2}\right)^{\frac{n-2}{2}},\quad y\in \mathbb{R}^n.
\end{equation*}
Recall that the space $L^2(\mathbb{S}^n)$ has an orthonormal basis $\Theta_m$, $m = 0, 1, \cdots, $ which are eigenfunctions of the problem
\begin{equation}\label{liftofkernelfunctions}
\Delta_{\mathbb{S}^n}\Theta_m + \lambda_m\Theta_m = 0 \quad\text{in}\quad\mathbb{S}^n
\end{equation}
so that
\begin{equation*}
0 = \lambda_0 < \lambda_1 = \cdots = \lambda_{n+1} = n < \lambda_{n+2} \leq \cdots.
\end{equation*}
One has $\Theta_0(y) = \alpha_0$ and $\Theta_j(y) = \alpha_1 y_j$, $j = 1,\cdots, n+1$, for constant numbers $\alpha_0$ and $\alpha_1$.
\begin{prop}\label{proposition5.2000}
Suppose $a \in (-n, -2)$, $\nu > 0$, $\|\tilde{h}\|_{a,\nu} < +\infty$ and
\begin{equation*}
\int_{\mathbb{S}^n}h(\tilde y,t)Z_j(\tilde y)d\tilde y = 0~\text{ for all }~t\in (t_0,\infty), ~j = 1,\cdots, n+1,
\end{equation*}
then there exists a function $\tilde \phi = \tilde \phi[\tilde{h}](\tilde y, t)$ satisfying (\ref{e:linearproblemonsphere}) and the estimate
\begin{equation*}
\begin{aligned}
&(\pi-|\tilde y|)|\nabla \tilde\phi(\tilde y, t)| + |\tilde\phi(\tilde y,t)|\lesssim t^{-\nu}(\pi-|\tilde y|)^{2+a}\|\tilde h\|_{a,\nu}.
\end{aligned}
\end{equation*}
\end{prop}
\begin{remark}
Here and in the following, $d\tilde y$ is the sphere measure on $\mathbb{S}^n$, and $|\tilde y|\in [0, \pi]$ is the geodesic distance of a point $\tilde y\in \mathbb{S}^n$ to the south pole $(0,\cdots, 0, -1)$, $\|\tilde h\|_{a, \nu}$ is least positive number $M$ such that
$$
|\tilde h(y, t)|\leq Mt^{-\nu}(\pi-|\tilde y|)^a.
$$
\end{remark}

\begin{lemma}\label{l5:3}
Suppose $a \in (-n, -2)$, $\nu > 0$, $\|\tilde h\|_{a,\nu} < +\infty$ and
\begin{equation*}
\int_{\mathbb{S}^n}\tilde h(\tilde y,t)\tilde Z_j(\tilde y)d\tilde y = 0~\text{ for all }~t\in (t_0,\infty), ~j = 1,\cdots, n+1.
\end{equation*}
Then, for any sufficiently large number $t_1 > 0$, the solution $(\phi(\tilde y,t), c(t))$ of the problem
\begin{eqnarray}\label{e5:32}
\left\{
\begin{aligned}
&\partial_t\tilde\phi = \left(\Delta_{\mathbb{S}^n} + \lambda_1\right)\tilde\phi + \tilde h -c(t)\tilde Z_0(\tilde y),~ \tilde y\in \mathbb{S}^n,~t\geq t_0,\\
&\int_{\mathbb{S}^n}\tilde \phi(\tilde y,t)\cdot\tilde Z_0(\tilde y)d\tilde y = 0~\mbox{ for all }~ t\in (t_0,+\infty),\\
&\tilde \phi(\tilde y,t_0) = 0,~\tilde y\in \mathbb{S}^n,
\end{aligned}
\right.
\end{eqnarray}
satisfies the estimates
\begin{equation}\label{e5:35}
\|\tilde \phi(\tilde y,t)\|_{a+2,t_1}\lesssim \|\tilde h\|_{a,t_1}
\end{equation}
and
\begin{equation*}
|c(t)|\lesssim t^{-\nu}\|\tilde h\|_{a,t_1}~ \text{ for }~t\in (t_0,t_1).
\end{equation*}
Here $\|\tilde h\|_{b,t_1}:=\sup_{t\in (t_0,t_1)}t^\nu\|(\pi-|\tilde y|)^{-b}\tilde h\|_{L^\infty(\mathbb{S}^n)}$.
\end{lemma}
\begin{proof}
Observe that (\ref{e5:32}) is equivalent to the following problem
\begin{eqnarray}\label{e5:99}
\left\{
\begin{aligned}
&\partial_t\tilde\phi = \left(\Delta_{\mathbb{S}^n} + \lambda_1\right)\tilde\phi + \tilde h -c(t)\tilde Z_0(\tilde y),~ \tilde y\in \mathbb{S}^n,~ t\geq t_0,\\
&\tilde \phi(\tilde y,t_0) = 0,~ \tilde y\in \mathbb{S}^n
\end{aligned}
\right.
\end{eqnarray}
for $c(t)$ determined by the relation
\begin{equation*}
c(t) \int_{\mathbb{S}^n}|\tilde Z_0(\tilde y)|^2d\tilde y =  \int_{\mathbb{S}^n}\tilde h(\tilde y,t)\cdot \tilde Z_0(\tilde y)d\tilde y.
\end{equation*}
Then it is easy to check that
\begin{equation}\label{e5:102}
|c(t)|\lesssim t^{-\nu}\|\tilde h\|_{a,t_1}
\end{equation}
holds for $t\in (t_0,t_1)$. Therefore we only need to show (\ref{e5:35}) for solutions $\tilde \phi$ of (\ref{e5:99}). We use the blowing-up arguments in the spirit of \cite{Davila2019}.

First, given $t_1 > t_0$, we have $\|\tilde \phi\|_{a+2, t_1} < +\infty$. Indeed, from the standard parabolic theory on sphere, given $R_0 \in (0, \pi)$, there exists a positive constant $K = K(R_0,t_1)$ such that
\begin{equation*}
|\tilde \phi(\tilde y,t)|\leq K \text{ in }\tilde B_{R_0}(0)\times (t_0, t_1].
\end{equation*}
Here $\tilde B_{R_0}(0)$ is the geodesic ball centered at the south pole with geodesic radius $R_0$. For a fixed $R_0$ close to $\pi$ and sufficiently large $K_1$, $K_1(\pi - |\tilde y|)^{2+a}$ is a super-solution of (\ref{e5:99}) when $|\tilde y| > R_0$. Therefore $|\tilde \phi|\leq 2K_1(\pi - |\tilde y|)^{2+a}$ and $\|\tilde \phi\|_{a+2,t_1} < +\infty$ holds for any $t_1 > 0$. Secondly, from the definition of $c(t)$, the following identities hold,
\begin{equation}\label{e5:33}
\int_{\mathbb{S}^n}\tilde \phi(\tilde y, t)\cdot \tilde Z_j(\tilde y)d\tilde y = 0\text{ for all }t\in (t_0,t_1), j= 0,1,\cdots, n+1.
\end{equation}
Finally, for any $t_1 > 0$ sufficiently large, and any $\tilde \phi$ satisfying (\ref{e5:99}), (\ref{e5:33}) and $\|\tilde \phi\|_{a+2,t_1} < +\infty$, we claim that the estimate
\begin{equation}\label{e5:34}
\|\tilde \phi\|_{a+2,t_1}\lesssim \|\tilde h\|_{a,t_1}
\end{equation}
holds, which implies (\ref{e5:35}).

To prove (\ref{e5:34}), we use contradiction arguments. Suppose there exist sequences $t_1^k\to +\infty$ and $\tilde\phi_k$, $\tilde h_k$, $c_k$ satisfying the equation
\begin{equation*}\label{e5:36}
\left\{
\begin{aligned}
&\partial_t\tilde \phi_k = \Delta_{\mathbb{S}^n}\tilde\phi_k + \lambda_1\tilde\phi_k + \tilde h_k - c_k(t)\tilde Z_0(\tilde y),~ \tilde y\in \mathbb{S}^n,~ t\geq t_0,\\
&\int_{\mathbb{S}^n}\tilde \phi_k(\tilde y, t)\cdot \tilde Z_j(\tilde y)d\tilde y = 0\text{ for all }t\in (t_0,t_1^k), ~j= 0, 1,\cdots, n+1,\\
&\tilde \phi_k(\tilde y,t_0) = 0, ~\tilde y\in \mathbb{S}^n
\end{aligned}
\right.
\end{equation*}
and also there hold
\begin{equation}\label{e5:38}
\|\tilde \phi_k\|_{a+2,t_1^k}=1,\quad \|\tilde h_k\|_{a,t_1^k}\to 0.
\end{equation}
From (\ref{e5:102}), $\sup_{t\in (t_0, t_1^k)}t^\nu c_k(t)\to 0$.
First, we claim it holds that
\begin{equation}\label{e5:37}
\sup_{t_0 < t < t_1^k}t^\nu|\tilde\phi_k(\tilde y,t)|\to 0
\end{equation}
uniformly on compact subsets away from the north point on $\mathbb{S}^n$. Indeed, if there are some points on $\mathbb{S}^n$ satisfying $|\tilde y_k|\leq M < \pi$ and $t_0 < t_2^k < t_1^k$,
\begin{equation*}
(t_2^k)^\nu(\pi - |\tilde y_k|)^{-a-2}|\tilde\phi_k(\tilde y_k,t_2^k)|\geq \frac{1}{2},
\end{equation*}
then we have $t_2^k\to +\infty$. Now let us define
\begin{equation*}
\bar{\phi}_k(\tilde y,t) = (t_2^k)^\nu\tilde\phi_k(\tilde y,t_2^k + t).
\end{equation*}
Then $\bar\phi_k$ is a solution of
\begin{equation*}
\partial_t\bar{\phi}_k = \Delta_{\mathbb{S}^n}\bar{\phi}_k + \lambda_1\bar\phi_k  + \bar{h}_k - \bar{c}_k(t)\tilde Z_0(\tilde y)\text{ in }\mathbb{S}^n\times (t_0-t_2^k,0],
\end{equation*}
with $\bar{h}_k\to 0$, $\bar{c}_k\to 0$ uniformly on compact subsets of $\left(\mathbb{S}^n\setminus\{\text{north pole}\}\right)\times (-\infty, 0]$, moreover, it holds that
\begin{equation*}
|\bar{\phi}_k(\tilde y,t)|\leq (\pi - |\tilde y|)^{a+2}\text{ in }\mathbb{S}^n\times (t_0-t_2^k,0].
\end{equation*}
From the dominant convergence theorem, we have $\bar{\phi}_k\to\bar{\phi}$ uniformly on compact subsets of $\left(\mathbb{S}^n\setminus\{\text{north pole}\}\right)\times (-\infty, 0]$, $\bar{\phi}\neq 0$ and satisfies the following equation
\begin{equation}\label{e5:103}
\left\{
\begin{aligned}
&\partial_t\bar{\phi} = \Delta_{\mathbb{S}^n}\bar{\phi} + \lambda_1\bar{\phi}~~\text{ in }~~\mathbb{S}^n\times (-\infty, 0],\\
&\int_{\mathbb{S}^n}\bar{\phi}(\tilde y, t)\cdot \tilde Z_j(\tilde y)d\tilde y = 0\text{ for all }t\in (-\infty, 0], ~ j= 0, 1,\cdots, n+1,\\
&|\bar{\phi}(\tilde y,t)|\leq (\pi - |\tilde y|)^{a+2}~\text{ in }~\mathbb{S}^n\times (-\infty, 0],\\
&\bar{\phi}(\tilde y,t_0) = 0, ~\tilde y\in \mathbb{S}^n.
\end{aligned}
\right.
\end{equation}
Now we claim that $\bar{\phi} = 0$, from which we obtain a contradiction. Indeed, by standard parabolic regularity on the sphere, $\bar{\phi}(\tilde y,t)$ is smooth. From a scaling argument, we have
\begin{equation*}
(\pi - |\tilde y|)|\nabla_{\mathbb{S}^n}\bar{\phi}| + |\bar{\phi}_t| + |\Delta_{\mathbb{S}^n}\bar{\phi}|\lesssim (\pi - |\tilde y|)^{2+a}.
\end{equation*}
Then differentiating (\ref{e5:103}) gives $\partial_t\bar{\phi}_t = \Delta_{\mathbb{S}^n}\bar{\phi}_t + \lambda_1\bar{\phi}_t$ and
\begin{equation*}
(\pi - |\tilde y|)|\nabla_{\mathbb{S}^n}\bar{\phi}_t| + |\bar{\phi}_{tt}| + |\Delta_{\mathbb{S}^n}\bar{\phi}_t|\lesssim (\pi - |\tilde y|)^{2+a}.
\end{equation*}
Furthermore, we have
\begin{equation*}
\frac{1}{2}\partial_t\int_{\mathbb{S}^n}|\bar{\phi}_t|^2 + B(\bar{\phi}_t, \bar{\phi}_t) = 0,
\end{equation*}
where
\begin{equation*}
B(\bar{\phi}, \bar{\phi}) = \int_{\mathbb{S}^n}\left[|\nabla_{\mathbb{S}^n}\bar{\phi}|^2 - \lambda_1|\bar{\phi}|^2\right]d\tilde y.
\end{equation*}
Since $\int_{\mathbb{S}^n}\bar{\phi}(\tilde y, t)\cdot \tilde Z_j(\tilde y)d\tilde y = 0$ for all $t\in (-\infty, 0]$, $j= 0, 1,\cdots, n+1 $, $B(\bar{\phi}, \bar{\phi})\geq 0$. Also, it holds that
\begin{equation*}
\int_{\mathbb{S}^n}|\bar{\phi}_t|^2 = -\frac{1}{2}\partial_t B(\bar{\phi}, \bar{\phi}).
\end{equation*}
From these relations, we obtain
\begin{equation*}
\partial_t\int_{\mathbb{S}^n}|\bar{\phi}_t|^2 \leq 0,\quad \int_{-\infty}^0dt\int_{\mathbb{S}^n}|\bar{\phi}_t|^2 < +\infty.
\end{equation*}
Therefore $\bar{\phi}_t = 0$, thus $\bar{\phi}$ is independent of $t$ and $\Delta_{\mathbb{S}^n}\bar{\phi}+\lambda_1\bar{\phi} = 0$. Since $\bar{\phi}$ is bounded, by the non-degeneracy of the elliptic operator $\Delta_{\mathbb{S}^n}+\lambda_1$, $\bar{\phi}$ can be expressed as a linear combination of the functions $\tilde Z_j$ defined in (\ref{liftofkernelfunctions}), $j = 1,\cdots, n+1$. But $\int_{\mathbb{S}^n}\bar{\phi}\cdot \tilde Z_j = 0$, $j = 1,\cdots, n$, we get $\bar{\phi} = 0$, which a contradiction. Therefore (\ref{e5:37}) holds.

From (\ref{e5:38}) and (\ref{e5:37}), there exists a sequence $\tilde y_k$ with $\pi - |\tilde y_k|\to 0$ such that
\begin{equation*}
(t_2^k)^\nu(\pi - |\tilde y_k|)^{-a-2}|\tilde\phi_k(\tilde y_k, t_2^k)|\geq \frac{1}{2}.
\end{equation*}
Let us write $\tilde \phi_k$ as a function of $\theta_1,\cdots, \theta_n$, i.e., $\tilde\phi_k=\tilde\phi_k(\theta_1,\cdots,\theta_n)$ with $\theta_n$ being the geodesic distance to the south pole. Suppose $\tilde y_k = (\theta_1^k,\cdots,\theta_n^k)$, then $\theta_n^k\to \pi$ and
\begin{equation*}
(t_2^k)^\nu(\pi - \theta_n^k)^{-a-2}|\tilde\phi_k(\theta_1^k,\cdots,\theta_n^k, t_2^k)|\geq \frac{1}{2}.
\end{equation*}
Set
\begin{equation}\label{e:20200523}
\begin{aligned}
\hat{\phi}_k(\vartheta_1,\cdots,\vartheta_n, t)&:=(t_2^k)^\nu(\pi - \theta_n^k)^{-a-2}\times\\
&\quad\quad\tilde\phi_k\left(\theta_1^k+(\pi - \theta_n^k)\vartheta_1,\cdots, \theta_n^k + (\pi - \theta_n^k)\vartheta_n, (\pi - \theta_n^k)^{2}t + t_2^k\right),
\end{aligned}
\end{equation}
then
\begin{equation*}
\partial_t \hat{\phi}_k = \Delta_{\mathbb{S}^n}\hat{\phi}_k + a_k\hat{\phi}_k + \hat{h}_k(z,t),
\end{equation*}
where
\begin{equation*}
\begin{aligned}
\hat{h}_k(\vartheta_1,\cdots,\vartheta_n, t) &: = (t_2^k)^\nu(\pi - \theta_n^k)^{-a}\times\\
&\quad\quad\tilde h_k\left(\theta_1^k+(\pi - \theta_n^k)\vartheta_1,\cdots, \theta_n^k + (\pi - \theta_n^k)\vartheta_n, (\pi - \theta_n^k)^{2}t + t_2^k\right).
\end{aligned}
\end{equation*}
From the assumptions on $h_k$, there holds
\begin{equation*}
|\hat{h}_k(\vartheta_1,\cdots,\vartheta_n, t)| \lesssim o(1)|1-\vartheta_n|^{a}((t_2^k)^{-1}(\pi - \theta_n^k)^{2}t + 1)^{-\nu}.
\end{equation*}
Thus $\hat{h}_k(\vartheta_1,\cdots,\vartheta_n, t)\to 0$ uniformly on compact subsets of $\mathbb{S}^n\times (-\infty, 0]$ and the function $a_k$ satisfies the same property. Furthermore, $|\tilde{\phi}_k(0,\cdots, 0)|\geq \frac{1}{2}$ and
\begin{equation*}
|\hat{\phi}_k| \lesssim |1-\vartheta_n|^{a+2}((t_2^k)^{-1}(\pi - \theta_n^k)^{2}t + 1)^{-\nu}.
\end{equation*}
Note that $|1-\vartheta_n|$ is the geodesic distance between the point $(\vartheta_1,\cdots,\vartheta_n)$ and $(\theta_1^k,\cdots,\theta_{n-1}^k, 1)$, by passing to a subsequence, we may assume $(\theta_1^k,\cdots,\theta_{n-1}^k, 1)\to \hat e\in \mathbb{S}^n$, the geodesic distance from $\hat e$ to the south pole is $1$. Hence there exists a function $\hat\phi$ such that $\hat{\phi}_k\to \hat{\phi}\neq 0$ uniformly on compact subsets of $\mathbb{S}^n\times (-\infty,0]$, and $\hat\phi$ satisfies the following equation
\begin{equation}\label{e5:39}
\hat{\phi}_t = \Delta_{\mathbb{R}^n}\hat{\phi}\quad\text{in }\mathbb{R}^n\times (-\infty,0]
\end{equation}
and
\begin{equation}\label{e5:40}
|\hat{\phi}(y,t)|\leq |y-e|^{a+2}\quad\text{in }\mathbb{R}^n\times (-\infty,0].
\end{equation}
Here $e$ is the pre-image of $\hat e$ under the stereographic projection, i.e., $\hat e = \pi(e)$ . Any functions satisfying (\ref{e5:39}), (\ref{e5:40}) and $a+2\in (2-n, 0)$ must be zero, which is a contradiction, hence we have the validity of (\ref{e5:34}).

Indeed, without loss of generality, assume $e$ is the origin point. Define $u(\rho, t) = (\rho^2+t)^{\frac{a+2}{2}} + \frac{\varepsilon}{\rho^{n-2}}$, then $-u_t+\Delta u < [-(a + 2) + \frac{1}{2}-(n-1)](\rho^2+t)^{\frac{a}{2}} < 0$, therefore $u(|y|, t+M)$ is a super-solution of (\ref{e5:39})-(\ref{e5:40}) on $\mathbb{R}^n\times [-M, 0]$. Hence we have $|\hat \phi(y, t)|\leq u(|y|, t+M)$. Letting $M\to +\infty$, we get $|\hat \phi(y, t)|\leq \frac{\varepsilon}{|y|^{n-2}}$. Since $\varepsilon$ is arbitrary, we conclude that $\hat\phi(y, t) = 0$.
\end{proof}
\begin{remark}
In (\ref{e:20200523}), if we define $\hat{\phi}_k$ as
\begin{equation*}
\begin{aligned}
\hat{\phi}_k(\vartheta_1,\cdots,\vartheta_n, t)&:=(t_2^k)^\nu(\pi - \theta_n^k)^{-a-2}\times\\
&\quad\tilde\phi_k\left(\theta_1^k+\vartheta_1,\cdots, \theta_{n-1}^k+\vartheta_{n-1}, \theta_n^k + (\pi - \theta_n^k)\vartheta_n, (\pi - \theta_n^k)^{2}t + t_2^k\right),
\end{aligned}
\end{equation*}
then the limit equation is
\begin{equation*}\label{e5:3900}
\hat{\phi}_t = \Delta_{\mathbb{S}^n}\hat{\phi}\quad\text{in }\mathbb{S}^n\times (-\infty,0]
\end{equation*}
and
\begin{equation*}\label{e5:4000}
|\hat{\phi}(y,t)|\leq |y-\hat e|^{a+2}\quad\text{in }\mathbb{S}^n\times (-\infty,0].
\end{equation*}
Under the assumption $a+2\in (2-n, 0)$ and similar arguments as above, one has $\hat \phi = 0$, which is also a contradiction.
\end{remark}

{\it Proof of Proposition \ref{proposition5.2000}}.
First, we consider the problem
\begin{equation*}
\left\{
\begin{aligned}
&\partial_t\tilde\phi = \left(\Delta_{\mathbb{S}^n} + \lambda_1\right)\tilde\phi + \tilde h -c(t)\tilde Z_0(\tilde y),~ \tilde y\in \mathbb{S}^n,~ t\geq t_0,\\
&\tilde \phi(\tilde y,t_0) = 0,~ \tilde y\in \mathbb{S}^n.
\end{aligned}
\right.
\end{equation*}
Let $(\tilde \phi(\tilde y,t),c(t))$ be the unique solution of the initial value problem (\ref{e5:32}), then by Lemma \ref{l5:3}, for any $t_1 > t_0$, there hold
\begin{equation*}
|\tilde\phi(\tilde y,t)|\lesssim t^{-\nu}(\pi-|\tilde y|)^{2+a}\|\tilde h\|_{a, t_1}, \text{ for all }t\in (t_0, t_1), \,\,\tilde y\in \mathbb{S}^n
\end{equation*}
and
\begin{equation*}
|c(t)|\leq t^{-\nu}\|\tilde h\|_{a,t_1}\text{ for all }t\in (t_0,t_1).
\end{equation*}
By assumption, we have $\|\tilde h\|_{a,\nu} < +\infty$ and $\|\tilde h\|_{a, t_1}\leq \|\tilde h\|_{a,\nu}$ for an arbitrary $t_1$. Therefore,
\begin{equation*}
|\tilde \phi(\tilde y,t)|\lesssim t^{-\nu}(\pi-|\tilde y|)^{2+a}\|\tilde h\|_{a,\nu}\text{ for all }t\in (t_0, t_1),\,\, \tilde y\in \mathbb{S}^n
\end{equation*}
and
\begin{equation*}
|c(t)|\leq t^{-\nu}\|\tilde h\|_{a,\nu}\text{ for all }t\in (t_0, t_1).
\end{equation*}
Since $t_1$ is arbitrary, we have
\begin{equation*}
|\tilde \phi(\tilde y,t)|\lesssim t^{-\nu}(\pi-|\tilde y|)^{2+a}\|\tilde h\|_{a,\nu}\text{ for all }t\in (t_0, +\infty),\,\, \tilde y\in \mathbb{S}^n
\end{equation*}
and
\begin{equation*}
|c(t)|\leq t^{-\nu}\|\tilde h\|_{a,\nu}\text{ for all }t\in (t_0, +\infty).
\end{equation*}
\qed

Using the stereographic projection, Proposition \ref{proposition5.2000} is equivalent to the following result.
\begin{prop}\label{proposition5.2000111}
Suppose $a \in (-n, -2)$, $\nu > 0$, $\|U^{p-1}h\|_{n+2+a,\nu} < +\infty$ and
\begin{equation}\label{orthogonalitycondition}
\int_{B_{2R}(0)}h(y,t)Z_j(y)U^{p-1}(y)d y = 0~\text{ for all }~t\in (t_0,\infty), ~j = 1,\cdots, n+1.
\end{equation}
Then, for sufficiently large $R$, there exists $\phi = \phi[h](y, t)$ satisfying (\ref{e:mainsection3}) and
\begin{equation*}
\begin{aligned}
&(1+|y|)^{-2}|\partial_t\phi(y, t)|+(1+|y|)^{-1}|\nabla \phi(y, t)| + |\phi(y,t)|\lesssim \frac{t^{-\nu}}{1+|y|^{n+a}}\|U^{p-1}h\|_{n+2+a,\nu}.
\end{aligned}
\end{equation*}
Furthermore, there exists a function $e_0 = e_0[h](t)$ such that
$\phi(\cdot, t_0) = e_0[h](t_0)Z_0(y)$ and $|e_0[h]|\lesssim \|U^{p-1}h\|_{n+2+a,\nu}$ hold.
\end{prop}

\subsection{Choice of the parameter functions}
To apply Proposition \ref{proposition5.2000111} to the inner problem (\ref{e:innerproblemselfsimilar}), the right hand term
\begin{equation*}
\begin{aligned}
H_j[\lambda, \xi, \dot\lambda, \dot\xi, \phi](y, t) & := p\mu_{0j}^{\frac{n-2}{2}}\frac{\mu_{0j}^2}{\mu^2_j}U^{p-1}\left(\frac{\mu_{0j}}{\mu_j}y\right)\psi(\xi_j + \mu_{0j} y, t)\\
&\quad + \mu_{0j}^{\frac{n+2}{2}}S^{*,in}_{\mu, \xi, j}(\xi_j + \mu_{0j} y, t) + B^1[\phi_j] + B^2[\phi_j] + B^3[\phi_j]
\end{aligned}
\end{equation*}
should satisfy the orthogonality conditions (\ref{orthogonalitycondition}), that is to say, we need the following identities
\begin{equation}\label{e:orthogonalitycondition1}
\int_{B_{2R}}H_j[\lambda, \xi, \dot\lambda, \dot\xi, \phi](y, t)Z_l(y)dy = 0\text{ for }l = 1, \cdots, n+1, ~~j= 1, 2, \cdots, k.
\end{equation}
These identities can be achieved by solving a system of ODEs for the parameter functions $\lambda_j$, $\xi_j$, $j=1,\cdots, k$.
\begin{lemma}\label{l5:1}
When $l = n+1$, identities (\ref{e:orthogonalitycondition1}) are equivalent to the following system of ODEs,
\begin{equation}\label{e5:11}
\dot{\lambda}_j + \frac{1}{t}\left(P^Tdiag\left(\frac{\bar\sigma + 2}{n-2}\right)P\lambda\right)_j = \Pi_{1, j}[\lambda, \xi, \dot{\lambda}, \dot{\xi}, \phi](t)
\end{equation}
where $\bar{\sigma}$ is a positive number and the right hand side term $\Pi_{1, j}[\lambda, \xi, \dot{\lambda}, \dot{\xi}, \phi](t)$ can be expressed as
\begin{equation}\label{expressionpi1}
\begin{aligned}
\Pi_{1, j}[\lambda, \xi, \dot{\lambda}, \dot{\xi}, \phi](t) =& t_0^{-\varepsilon}\mu_0^{n-1 + \sigma}(t)f_j(t)\\
& + t_0^{-\varepsilon}\Theta_{1, j}\left[\dot{\lambda},\dot{\xi}, \mu_0^{n-2}(t)\lambda, \mu_0^{n-2}(\xi-q), \mu_0^{n-1+\sigma}\phi\right](t).
\end{aligned}
\end{equation}
Here $f_j(t)$ and $\Theta_{1, j}\left[\dot{\lambda},\dot{\xi}, \mu_0^{n-2}(t)\lambda, \mu_0^{n-2}(\xi-q), \mu_0^{n-1+\sigma}\phi\right](t)$ are smooth bounded functions of $t$. Furthermore, the following Lipschitz properties hold,
\begin{equation*}
\left|\Theta_{1, j}[\dot{\lambda}_1](t) - \Theta_{1, j}[\dot{\lambda}_2](t)\right|\lesssim t_0^{-\varepsilon}|\dot{\lambda}_1(t) - \dot{\lambda}_2(t)|
\end{equation*}
\begin{equation*}
\left|\Theta_{1, j}[\dot{\xi}_1](t) - \Theta_{1, j}[\dot{\xi}_2](t)\right|\lesssim t_0^{-\varepsilon}|\dot{\xi}_1(t) - \dot{\xi}_2(t)|,
\end{equation*}
\begin{equation*}
\left|\Theta_{1, j}[\mu_0^{n-2}\lambda_1](t) - \Theta_{1, j}[\mu_0^{n-2}\lambda_2](t)\right|\lesssim t_0^{-\varepsilon}|\dot{\lambda}_1(t) - \dot{\lambda}_2(t)|
\end{equation*}
\begin{equation*}
\left|\Theta_{1, j}[\mu_0^{n-2}(\xi_1-q)](t) - \Theta_{1, j}[\mu_0^{n-2}(\xi_2-q)](t)\right|\lesssim t_0^{-\varepsilon}|\xi_1(t) - \xi_2(t)|,
\end{equation*}
\begin{equation*}\label{e5:104}
\left|\Theta_{1, j}[\mu_0^{n-1+\sigma}\phi_1](t) - \Theta_{1, j}[\mu_0^{n-1+\sigma}\phi_2](t)\right|\lesssim t_0^{-\varepsilon}\|\phi_1(t) - \phi_2(t)\|_{n-2+\sigma, n+a}.
\end{equation*}
\end{lemma}
\begin{proof}
For a fixed $j \in \{1,\cdots,k\}$, let us compute the term
\begin{eqnarray*}
\int_{B_{2R}}H_j[\lambda,\xi,\dot{\lambda},\dot{\xi}, \phi](y,t)Z_{n+1}(y)dy.
\end{eqnarray*}
First, we consider the term
\begin{equation*}
\begin{aligned}
&\mu_{0j}^{\frac{n+2}{2}}S^{*,in}_{\mu, \xi, j}(\xi_j + \mu_{0j}y, t)\\
&=\left(\frac{\mu_{0j}}{\mu_j}\right)^{\frac{n+2}{2}}\left[\mu_{0j}^{-1}S_1(z, t) + \lambda_jb_j^{-1}\mu_0^{-2}S_2(z, t) + \mu_j^{-2}S_3(z, t)\right]_{z = \xi_j + \mu_jy}\\
&\quad+\left(\frac{\mu_{0j}}{\mu_j}\right)^{\frac{n+2}{2}}\mu_{0j}\mu_0^{-2}\left[S_1(\xi_j+\mu_{0j}y, t)-S_1(\xi_j+\mu_{j}y, t)\right]\\
&\quad+\left(\frac{\mu_{0j}}{\mu_j}\right)^{\frac{n+2}{2}}\lambda_jb_j^{-1}\mu_0^{-2}\left[S_2(\xi_j+\mu_{0j}y, t)-S_2(\xi_j+\mu_{j}y, t)\right]\\
&\quad+\left(\frac{\mu_{0j}}{\mu_j}\right)^{\frac{n+2}{2}}\mu_j^{-2}\left[S_3(\xi_j+\mu_{0j}y, t)-S_3(\xi_j+\mu_{j}y, t)\right],
\end{aligned}
\end{equation*}
where
\begin{equation*}
\begin{aligned}
S_1(z) &= \dot{\lambda}_jpU(\frac{z-\xi_j}{\mu_j})^{p-1}Z_{n+1}\left(\frac{z-\xi_j}{\mu_j}\right)\\
&\quad -2\mu_{0}^{-1}\dot\mu_0\lambda_jpU\left(\frac{z-\xi_j}{\mu_j}\right)^{p-1}Z_{n+1}\left(\frac{z-\xi_j}{\mu_j}\right)\\
&\quad -\mu_0^{n-2}pU(\frac{z-\xi_j}{\mu_j})^{p-1}\sum_{i=1}^kb_j^2\mathcal M_{ij}\lambda_i,
\end{aligned}
\end{equation*}
\begin{equation*}
\begin{aligned}
S_2(z) =& \dot{\mu}_0Z_{n+1}\left(\frac{z-\xi_j}{\mu_j}\right)pU\left(\frac{z-\xi_j}{\mu_j}\right)^{p-1}\\
&+pU\left(\frac{z-\xi_j}{\mu_j}\right)^{p-1}\mu_0^{n-1}\bigg(-b_j^{n-2}H(q_j, q_j)+\sum_{i\neq j}b_j^{\frac{n-2}{2}}b_i^{\frac{n-2}{2}}G(q_j, q_i)\bigg)
\end{aligned}
\end{equation*}
and
\begin{equation*}
\begin{aligned}
S_3(z) =& \dot{\xi}_j\cdot \nabla U\left(\frac{z-\xi_j}{\mu_j}\right)+ \mu_j^3pU\left(\frac{z-\xi_j}{\mu_j}\right)^{p-1}\\
&\times\left(-\mu_j^{n-2}\nabla H(q_j, q_j) + \sum_{i\neq j}\mu_j^{\frac{n-2}{2}}\mu_i^{\frac{n-2}{2}}\nabla G(q_j, q_i)\right)\cdot \left(\frac{z-\xi_j}{\mu_j}\right).
\end{aligned}
\end{equation*}
By direct computations, we have
\begin{equation*}
\begin{aligned}
\int_{B_{2R}}& S_1(\xi_j+\mu_jy)Z_{n+1}(y)dy = c_2(1+O(R^{2-n}))\dot{\lambda}_j \\
&\quad  -2c_2(1+O(R^{-2}))\mu_{0}^{-1}\dot\mu_0\lambda_j + c_1(1+O(R^{-2}))\mu_0^{n-2}\sum_{i=1}^kb_j^2\mathcal M_{ij}\lambda_i,
\end{aligned}
\end{equation*}
\begin{equation*}
\begin{aligned}
&\int_{B_{2R}}S_2(\xi_j+\mu_jy)Z_{n+1}(y)dy = O(R^{2-n}+R^{-2})\mu_{0}^{n-1}
\end{aligned}
\end{equation*}
and
\begin{equation*}
\int_{B_{2R}}S_3(\xi_j+\mu_jy)Z_{n+1}(y)dy = 0 \,\,(\text{by symmetry}).
\end{equation*}
Since $\frac{\mu_{0j}}{\mu_j} = (1 + \frac{\lambda_j}{\mu_{0j}})^{-1}$, for any $l = 1, 2, 3$, there holds
\begin{equation*}
\begin{aligned}
&\int_{B_{2R}}[S_l(\xi_j+\mu_{0j}y, t)-S_l(\xi_j+\mu_{j}y, t)]Z_{n+1}(y)dy \\
&= g(t,\frac{\lambda_j}{\mu_0})\dot{\lambda}_j + g(t, \frac{\lambda_j}{\mu_0})\dot{\xi} +g(t,\frac{\lambda_j}{\mu_0})\sum_i\mu_0^{n-2}\lambda_i + \mu_0^{n-1+\sigma}f(t),
\end{aligned}
\end{equation*}
where $f$, $g$ are smooth bounded functions satisfying $f(\cdot, s)\sim s$, $g(\cdot, s)\sim s$ as $s\to 0$. Therefore we have
\begin{equation*}
\begin{aligned}
& c\left(\frac{\mu_{j}}{\mu_{0j}}\right)^{\frac{n+2}{2}}\mu_{0j}\int_{B_{2R}}\mu_{0j}^{\frac{n+2}{2}}S_{\mu, \xi, j}(\xi_j + \mu_{0j}y, t)Z_{n+1}(y)dy\\
& = \left[\dot{\lambda}_j + \frac{1}{t}\left(P^Tdiag\left(\frac{\bar\sigma_j + 2}{n-2}\right)P\lambda\right)_j\right] + t_0^{-\varepsilon}g(t,\frac{\lambda_j}{\mu_0})(\dot{\lambda} + \dot{\xi}) + t_0^{-\varepsilon}\mu_0^{n-2}g(t, \frac{\lambda_j}{\mu_0}),
\end{aligned}
\end{equation*}
for a positive number $c$. Here we have used the fact that, since $\mathcal G(q)$ is positive definite, the matrix with elements $\frac{1}{2}b_j^2\mathcal M_{ij}$ can be diagonalized as $\frac{1}{n-2}P^T\left(\bar{\sigma}_1,\cdots, \bar{\sigma}_k\right)P$ with $\bar{\sigma}_i > 0$ for $i=1, \cdots, k$ and $P$ is a $k\times k$ matrix.

Next we compute the term $$p\mu_{0j}^{\frac{n-2}{2}}(1 + \frac{\lambda_j}{\mu_{0j}})^{-2}\int_{B_{2R}}U^{p-1}(\frac{\mu_{0j}}{\mu_j}y)\psi(\xi_j + \mu_{0j}y, t)Z_{n+1}(y)dy,$$
the principal part is $I: = \int_{B_{2R}}U^{p-1}(y)\psi(\xi_j + \mu_{0j}y, t)Z_{n+1}(y)dy$. Decompose $I$ as
\begin{equation*}
\begin{aligned}
I &=\psi[0,q,0,0,0](q_j, t)\int_{B_{2R}}U^{p-1}(y)Z_{n+1}(y)dy\\
&\quad + \int_{B_{2R}}U^{p-1}(y)Z_{n+1}(y)(\psi[0,q,0,0,0](\xi_j+\mu_{0j}y, t)-\psi[0,q,0,0,0](q_j, t))dy\\
&\quad + \int_{B_{2R}}U^{p-1}(y)Z_{n+1}(y)(\psi[\lambda,\xi,\dot{\lambda},\dot{\xi}, \phi] - \psi[0,q,0,0,0])(\xi_j+\mu_{0j}y, t)dy\\
& = I_1 + I_2 + I_3.
\end{aligned}
\end{equation*}
By Proposition \ref{propositionouterproblem}, $I_1 = t_0^{-\varepsilon}\mu_0^{\frac{n-2}{2}+\sigma}f(t)$, $f$ is a smooth bounded function. Similarly, $I_2 = t_0^{-\varepsilon}\mu_0^{\frac{n-2}{2}+\sigma}g(t, \frac{\lambda}{\mu_0}, \xi -q)$ for a bounded function $g$ such that $g(\cdot, s, \cdot)\sim s$ and $g(\cdot,\cdot, s)\sim s$ as $s \to 0$. From Remark \ref{propositionouterproblem1} and mean value theorem, $I_3$ is the sum of terms like
\begin{equation*}
\mu_0^{-\frac{n}{2}+\sigma}t_0^{-\varepsilon}f(t)(\dot{\lambda} + \dot{\xi})F[\lambda, \xi, \dot{\lambda}, \dot{\xi}, \phi](t)
\end{equation*}
and
$$
\mu_0^{\frac{n-4}{2}}t_0^{-\varepsilon}f(t)(\lambda + \xi)F[\lambda, \xi, \dot{\lambda}, \dot{\xi}, \phi](t),
$$
where the function $f$ is smooth bounded, $F$ is a nonlocal operator with $F[0,q,0,0,0](t)$ bounded.

Finally, there hold
\begin{equation*}
\int_{B_{2R}}B^i[\phi_j](y, t)Z_{n+1}(y)dy = t_0^{-\varepsilon}[\mu_0^{n-1+\sigma}(t)g^i[\phi](t) + \dot{\xi}_j\ell^i[\phi](t)]
\end{equation*}
for functions $g^i(s)$ satisfying $g^i(s)\sim s$ as $s\to 0$, and $\ell^i[\phi](t)$ is smooth bounded in $t$.
Combining all the estimates above, we conclude the result.
\end{proof}
\noindent Similarly, for the identities
\begin{equation*}
\int_{B_{2R}}H_j[\lambda,\xi, \dot{\lambda}, \dot{\xi},\phi](y, t(t))Z_l(y)dy,
\end{equation*}
for any $j = 1, \cdots, k$, $l = 1,\cdots, n$, we have
\begin{lemma}\label{l5:2}
For $j = 1, \cdots, k$, $l = 1,\cdots, n$, (\ref{e:orthogonalitycondition1}) are equivalent to the following system of ODEs
\begin{equation*}\label{e5:18}
\dot{\xi}_j = \Pi_{2, j}[\lambda, \xi, \dot{\lambda}, \dot{\xi}, \phi](t),
\end{equation*}
\begin{equation*}
\begin{aligned}
&\Pi_{2, j}[\lambda, \xi, \dot{\lambda}, \dot{\xi}, \phi](t)\\
&= \mu_0^{n}c\left[b_j^{n-2}\nabla H(q_j, q_j) - \sum_{i\neq j}b_j^{\frac{n-2}{2}}b_i^{\frac{n-2}{2}}\nabla G(q_j, q_i)\right]+ \mu_0^{n+\sigma}(t)f_j(t)\\
&\quad + t_0^{-\varepsilon}\Theta_{2, j}[\dot{\lambda},\dot{\xi}, \mu_0^{n-2}(t)\lambda, \mu_0^{n-1}(\xi-q), \mu_0^{n-1+\sigma}\phi](t),
\end{aligned}
\end{equation*}
where $c = \frac{p\int_{\mathbb{R}^n}U^{p-1}\frac{\partial U}{\partial y_1}y_1dy}{\int_{\mathbb{R}^n}\left(\frac{\partial U}{\partial y_1}\right)^2dy}$, $f_j(t)$ are smooth bounded $n$ dimensional vector functions for $t\in [t_0, \infty)$, the $n$ dimensional vector functions $\Theta_{2, j}$ has the same properties as in Lemma \ref{l5:1}.
\end{lemma}

\noindent From Lemma \ref{l5:1} and Lemma \ref{l5:2}, we know that the orthogonality conditions
\begin{equation*}
\int_{B_{2R}}H_j[\lambda,\xi, \dot{\lambda}, \dot{\xi},\phi](y, t(t))Z_l(y)dy, \mbox{ for } j = 1,\cdots, k \mbox{ and } l = 1, \cdots, n+1,
\end{equation*}
are equivalent to the system of ODEs for $\lambda$ and $\xi$,
\begin{equation}\label{e5:9}
\left\{
\begin{aligned}
&\dot{\lambda}_j + \frac{1}{t}\left(P^Tdiag\left(\frac{\bar{\sigma} + 1}{n-2}\right)P\lambda\right)_j = \Pi_{1, j}[\lambda, \xi, \dot{\lambda}, \dot{\xi}, \phi](t),\\
&\dot{\xi}_j = \Pi_{2, j}[\lambda, \xi, \dot{\lambda}, \dot{\xi}, \phi](t), ~~ j = 1,\cdots, k.
\end{aligned}
\right.
\end{equation}
System (\ref{e5:9}) is solvable for $\lambda$ and $\xi$ satisfying (\ref{e4:210})-(\ref{e4:220}). Indeed, we have
\begin{prop}\label{p5:5.1}
There exists a solution $\lambda = \lambda[\phi](t)$, $\xi = \xi[\phi](t)$ to (\ref{e5:9}) satisfying
\begin{equation*}
|\lambda[\phi_1](t)-\lambda[\phi_2](t)|\lesssim t_0^{-\varepsilon}\mu_0^{1+\sigma}\|\phi_1-\phi_2\|_{n-2+\sigma, n+a}
\end{equation*}
and
\begin{equation*}
|\xi[\phi_1](t)-\xi[\phi_2](t)|\lesssim t_0^{-\varepsilon}\mu_0^{1+\sigma}\|\phi_1-\phi_2\|_{n-2+\sigma, n+a}.
\end{equation*}
\end{prop}
\noindent The proof is similar to that of \cite{delPinoMussoJEMS}, so we omit it here.

\subsection{Solving the inner problem (\ref{e:innerproblemselfsimilar})}
After the parameter functions $\lambda = \lambda[\phi]$ and $\xi = \xi[\phi]$ have been chosen such that the orthogonality conditions (\ref{e:orthogonalitycondition1}) hold, problem (\ref{e:innerproblemselfsimilar}) can be solved in the class of functions satisfying $\|\phi\|_{n-2+\sigma,n+a} < +\infty$ bounded. From Proposition \ref{proposition5.2000111}, there exists a bounded linear operator $\mathcal{T}$ associating any function $h(y,t)$ with $\|U^{p-1}h\|_{n-2+\sigma, n+2+a}$-bounded the solution of problem (\ref{e:mainsection3}), thus (\ref{e:innerproblemselfsimilar}) reduces to the following fixed point problem
\begin{equation*}
\phi = (\phi_1,\cdots, \phi_k) = \mathcal{A}(\phi): = (\mathcal{T}(H_1[\lambda,\xi,\dot{\lambda},\dot{\xi},\phi]), \cdots, \mathcal{T}(H_k[\lambda,\xi,\dot{\lambda},\dot{\xi},\phi])).
\end{equation*}

From the definition of $H_j$, we have the estimate
\begin{equation}\label{e6:1}
\begin{aligned}
&\left|H[\lambda,\xi,\dot{\lambda},\dot{\xi},\phi](y, t)\right|\lesssim t_0^{-\varepsilon}\frac{\mu_0^{n-2+\sigma}}{1+|y|^{n+2+a}}.
\end{aligned}
\end{equation}
Therefore $\mathcal{A}$ maps the set $\Lambda:=\{\phi\,|\, \|\phi\|_{n-2+\sigma, n+a}\leq ct_0^{-\varepsilon}\}$ into itself for some large constant $c> 0$.

Moreover, $\mathcal{A}$ is a contraction map, hence there exists a fixed point, from which we find a solution of (\ref{e:main}). Indeed, this is consequence of the following estimates:

\noindent {(a)} \begin{equation*}
\begin{aligned}
&\mu_{0j}^{\frac{n+2}{2}}\left|S_{\mu_1, \xi_1, j}(\xi_{j,1} + \mu_{0j}y, t) - S_{\mu_2, \xi_2, j}(\xi_{j,2} + \mu_{0j}y, t)\right|\\
&\quad\quad\quad\quad\quad\quad\quad\quad\quad\quad\quad\quad\quad\quad\quad\quad\quad\quad\lesssim t_0^{-\varepsilon}\frac{\mu_0^{n-2+\sigma}(t)}{1+|y|^{n+2+a}}\|\phi^{(1)} - \phi^{(2)}\|_{n-2+\sigma, n+a}
\end{aligned}
\end{equation*}
where
\begin{equation*}
\mu_i = \mu[\phi^{(i)}],\quad \xi_i = \xi[\phi^{(i)}],\quad \xi_{j, i} = \xi_j[\phi^{(i)}],\quad i = 1, 2.
\end{equation*}

\noindent {(b)}
From Remark \ref{propositionouterproblem1}, we have
\begin{eqnarray*}
\begin{aligned}
& p\mu_{0j}^{\frac{n-2}{2}}\Bigg|\frac{\mu_{0j}^{2}}{\mu_{j,1}^{2}}U^{p-1}\left(\frac{\mu_{0j}}{\mu_{j,1}}y\right)\psi[\phi^{(1)}](\xi_{j,1} + \mu_{0j}y, t)\\
&\quad\quad\quad\quad\quad\quad\quad\quad\quad\quad\quad\quad\quad\quad\quad\quad -\frac{\mu_{0j}^{2}}{\mu_{j,2}^{2}}U^{p-1}\left(\frac{\mu_{0j}}{\mu_{j,2}}y\right)\psi[\phi^{(2)}](\xi_{j,2} + \mu_{0j}y, t)\Bigg|\\
&\quad\quad\quad\quad\quad\quad\quad\quad\quad\quad\quad\quad\quad\quad\quad\quad\quad\lesssim t_0^{-\varepsilon}\frac{\mu_0^{n-2+\sigma}(t)}{1+|y|^{n+2+a}}\|\phi^{(1)} - \phi^{(2)}\|_{n-2+\sigma, n+a}
\end{aligned}
\end{eqnarray*}
where
\begin{equation*}
\mu_{j,i} = \mu_j[\phi^{(i)}],\quad \psi[\phi^{(i)}] = \Psi[\lambda_i, \xi_i, \dot{\lambda}_i, \dot{\xi}_i, \phi^{(i)}],\quad i = 1, 2.
\end{equation*}

\noindent {(c)}
From the definitions in Section 2, we have
\begin{equation*}
\left|B^l[\phi^{(1)}_j]-B^{(1)}_j[\phi^{(2)}_j]\right|\lesssim t_0^{-\varepsilon}\frac{\mu_0^{n-2+\sigma}(t)}{1+|y|^{n+2+a}}\|\phi^{(1)} - \phi^{(2)}\|_{n-2+\sigma, n+a}
\end{equation*}
hold for $l=1, 2, 3$.\qed
\section*{Acknowledgements}
J. Wei is partially supported by NSERC of Canada.

\end{document}